\documentclass[a4paper,reqno,11pt]{amsart}

\usepackage{mathtools}
\usepackage{amsfonts}
\usepackage{amssymb}
\usepackage{graphicx,subcaption}
\usepackage[shortlabels]{enumitem}
\usepackage{csquotes}
\usepackage{microtype}
\usepackage{xcolor}

\usepackage{trimclip}
\makeatletter
\DeclareRobustCommand{\shortto}{%
  \mathrel{\mathpalette\short@to\relax}%
}

\newcommand{\short@to}[2]{%
  \mkern2mu
  \clipbox{{.5\width} 0 0 0}{$\m@th#1\vphantom{+}{\rightarrow}$}%
  }
\makeatother

\newtheorem{theorem}{Theorem}[section]
\newtheorem{lemma}[theorem]{Lemma}
\newtheorem{proposition}[theorem]{Proposition}

\theoremstyle{definition}

\theoremstyle{remark}

\newtheorem{remark}[theorem]{Remark}

\numberwithin{equation}{section}
\numberwithin{table}{section}
\numberwithin{figure}{section}

\newcommand\Ab{\mathbf{A}}

\newcommand\Xb{\mathbf{X}}

\newcommand\laplace{S}

\newcommand\DN{\mathrm{DN}}
\renewcommand\i{\mathrm{i}}

\DeclareMathOperator{\Mor}{Mor}

\DeclareMathOperator{\Int}{Int}

\DeclareMathOperator{\curl}{curl}
\DeclareMathOperator{\Def}{Def}
\let\phi=\varphi
\let\epsilon=\varepsilon

\title[Spectral flow for pair compatible equipartitions]{Spectral flow for\\pair compatible equipartitions}

\author{Bernard Helffer}
\address[Bernard Helffer]{ Laboratoire de Math\'ematiques Jean Leray,
Universit\'e de Nantes, 44332 Nantes,  France.}
\email{Bernard.Helffer@math.univ-nantes.fr}
\author{Mikael Persson Sundqvist}
\address[Mikael Persson Sundqvist]{Lund University, Department of Mathematical 
Sciences, Box 118, 221\ 00 Lund, Sweden.}
\email{mikael.persson\_sundqvist@math.lth.se}


\begin{document}

\begin{abstract}
We show that a recent spectral flow approach  proposed by Berkolaiko--Cox--Marzuola   for analyzing the nodal deficiency  of the nodal partition associated to an eigenfunction  can be extended to  more general 
partitions. To be more precise, we work with spectral
equipartitions that satisfy a pair compatible condition. Nodal partitions and
spectral minimal partitions are examples of such partitions.

Along the way, we discuss different approaches to the Dirichlet-to-Neumann 
operators: via Aharonov--Bohm operators, 
via a double covering argument, and via a slitting of the domain.
\end{abstract}

\maketitle

\section{Introduction}

\subsection{Main goals}
We consider the Dirichlet Laplacian $\laplace_\Omega=-\Delta$ in a bounded
domain $\Omega\subset \mathbb R^2$ (and subdomains of $\Omega$), where 
$\partial\Omega$ is assumed to be piecewise differentiable.

We would like to analyze the  relations between the nodal domains of the real-valued  eigenfunctions of this Laplacian and the partitions $\mathcal D$ of $\Omega$ by $k$ open sets $\{D_i\}_{i=1}^k$, which are spectral equipartitions in the sense that in each $D_i$'s the ground state energy $\lambda_1(D_i)$ of the Dirichlet realization of the Laplacian $\laplace_{D_i}$ in $ D_i$ is the same. In addition we will consider spectral equipartitions which satisfy a pair compatibility condition (PCC) for any pair of neighbouring $D_i$'s,  i.e. for any pair of neighbors $D_i,D_j$ in $\mathcal D$, there is a linear combination of the ground states in $D_i$ and $D_j$ which is an eigenfunction of the Dirichlet problem in $\Int(\overline{D_i\cup D_j})$.

Nodal partitions and minimal partitions are typical examples of these PCC-equipartitions but a difficult question is to recognize which  PCC-equipartitions are minimal. This problem has been solved in the bipartite case (which corresponds to the Courant sharp situation) but the problem remains open in the general case.

Our main goal is to extend the construction and analysis of spectral flow and Dirichlet-to-Neumann operators, which was done for nodal partitions in~\cite{BeCoMa}, to spectral equipartitions that satisfy the PCC. We describe briefly the construction for nodal domains first, and return to our setting in Section~\ref{sec:ourresults}.

\subsection{The spectral flow construction by Berkolaiko--Cox--Marzuola}
We describe shortly the result in~\cite{BeCoMa} that we want to generalize,  together with their construction.

Let $\Omega\subset\mathbb R^2$ and let $\lambda_*$ be some eigenvalue of the Dirichlet Laplacian $\laplace_\Omega=-\Delta$, with corresponding eigenfunction $\phi_*$. We denote by $\Gamma$ the nodal set of $\phi_*$ inside $\Omega$, i.e. $\Gamma=\{x\in\Omega~:~\phi_*(x)=0\}$, and by $k=\mu(\phi_*)$  the number of nodal domains of $\phi_*$, i.e. the number of connected components of the set $\{x\in\Omega~:~\phi_*(x)\neq 0\}$. We denote by $D_i$ ($i= 1,\dots, k$) the nodal domains of $\phi_*$. Also, let $k_*$ be the label of the eigenvalue $\lambda_*$ if it is simple and the minimal label if $\lambda_*$ is degenerate.

To state the main result of~\cite{BeCoMa}, we need to introduce Dirichlet-to-Neumann operators. We only do this at an intuitive level at this point, and refer the reader to~\cite{AM} for more details. Assume that $E\subset\mathbb R^2$ is a bounded domain, and that $\lambda$ is not in the spectrum of the Dirichlet Laplacian $\laplace_E$ on $E$. Given a sufficiently regular function $g$ on $\partial E$, let $u$ be the unique solution to
\begin{equation}\label{eq:defDN}
\begin{cases}
-\Delta u=\lambda u& \text{in }E,\\
u=g & \text{on }\partial E.
\end{cases}
\end{equation}
Then, the Dirichlet-to-Neumann operator $\DN_E(\lambda)\colon L^2(\partial E)\to L^2(\partial E)$ is defined by
\[
\DN_E(\lambda)g:=\frac{\partial u}{\partial \nu},
\]
where $\nu$ is a unit normal vector pointing out of $E$. Again, the reader is referred to \cite[Section~2]{AM} for more details. We denote by $R_{\Gamma\shortto\partial D_i}:L^2(\Gamma)\to L^2(\partial D_i)$ the operator that first extends by $0$ on $\partial\Omega$ to get a function on $\Gamma \cup \partial \Omega$  and then restricts to $\partial D_i$, and we denote by $R_{\partial D_i\shortto\Gamma}\colon L^2(\partial D_i)\to L^2(\Gamma)$ the extension by $0$ operator to $\Gamma\cup \partial \Omega$ composed by the restriction operator from $L^2(\Gamma\cup\partial\Omega)$ to $L^2(\Gamma)$. We can then introduce the  $\Gamma$-Dirichlet-to-Neumann operator $\DN(\Gamma,\lambda)$ by 
\[
  \DN(\Gamma,\lambda)
  = \sum_{i=1}^k   R_{\partial D_i\shortto\Gamma}  \DN_{D_i} (\lambda)  R_{\Gamma \shortto\partial D_i},
\]
considered as an unbounded operator in $L^2(\Gamma)$. This is (for a suitable $\lambda$, see below) the operator  $\Lambda_+(\epsilon)+\Lambda_-(\epsilon)$ introduced  in \cite{BeCoMa} but we have preferred this presentation which is easier to generalize. Formally, the operator is defined via the quadratic form with form domain $H^{1/2}(\Gamma)$.

\begin{theorem}[{\cite{CJM,BeCoMa}}]\label{thm:BeCoMa}
If $\epsilon>0$ is sufficiently small, then
\begin{equation}\label{eq:BeCoMa}
k_*-\mu(\phi_*)
=1-\dim\ker(\laplace_\Omega-\lambda_*)
+\Mor\bigl(\DN(\Gamma,\lambda_*+\epsilon)\bigr),
\end{equation}
where $\Mor$ counts the number of negative eigenvalues of an operator (the so-called Morse index of the operator).
\end{theorem}

\begin{remark}
The number $k_*-\mu(\phi_*)$ in the left-hand side above is non-negative due to Courant's nodal theorem. It is usually called the \emph{nodal deficiency} of the eigenfunction $\phi_*$ (see for example~\cite{BeCoMa}). If $\lambda_*$ is a simple eigenvalue of $\laplace_\Omega$ then the right-hand side above is non-negative, and an independent argument for Courant's theorem is provided.
\end{remark}

It turns out, that to characterize the negative eigenvalues of the operator $\DN(\Gamma,\lambda_*+\epsilon)$ it is fruitful to study the family of operators $\laplace_{\Omega,\sigma}$, $0\leq\sigma<+\infty$, induced by the bilinear form
\[
\mathfrak s_\sigma(u,v)=\int_\Omega \nabla u\cdot\nabla v\,dx+\sigma\int_\Gamma u\,v\,ds,
\quad u,v\in H_0^1(\Omega).
\]
Also, let $\laplace_{\Omega,+\infty}$ be the Laplacian in $\Omega$ with Dirichlet boundary conditions imposed on $\partial\Omega\cup\Gamma$. Indeed, if we denote by $\{\lambda_k(\sigma)\}_{k=1}^{+\infty}$ the set of eigenvalues of $\laplace_{\Omega,\sigma}$, in increasing order, then Berkolaiko--Cox--Marzuola shows  by following the analytic branches of the eigenvalues of $\laplace_{\Omega,\sigma}$, that if $\epsilon>0$ is sufficiently small, then $-\sigma$ is an eigenvalue of $\DN(\Gamma,\lambda_*+\epsilon)$ if, and only if, $\lambda_*+\epsilon=\lambda_k(\sigma)$ for some $k\in\mathbb N$.

\subsection{New results}\label{sec:ourresults}
The results we obtain look very similar to the one in Theorem~\ref{thm:BeCoMa}, but with slightly different terms. Let us discuss here, on a formal level, the ingredients, and refer to the places where the objects are introduced more carefully.

The main difference in our approach is that we will consider a specific  magnetic Schr\"odinger  operator instead of the Laplace operator. The eigenfunctions corresponding to the magnetic Schrödinger  operator do not usually divide the domain $\Omega$ into nodal domains, since they are in general complex-valued. It has been seen, though, that one obtains a similar partition of $\Omega$ by the nodal domains of  the \enquote{real} eigenfunctions of  a  \enquote{special}   magnetic Aharonov--Bohm Laplacian associated with a magnetic field consisting of  sum of delta distributions  with coefficients $\pi$. This was first observed by Berger--Rubinstein~\cite{BeRu}, more systematically developed in~\cite{HHOO} and the link with minimal spectral partitions appears later quite useful, theoretically~\cite{HH:2005a,HHOT, HH:2015b} and numerically~\cite{BHHO}.

Nodal partitions and minimal partitions share the property to be regular equi-partitions and to satisfy the Pair Compatibility Condition (PCC). Precise definitions will be given in Section~\ref{sec:equipartitions}, but say simply here that we are given a $k$-partition $\mathcal D$  of $\Omega$ with disjoint open sets $D_i$ such that $\overline{\cup D_i} =\overline{\Omega}$ whose boundary  set $\Gamma = \Omega \cap (\cup \partial D_i)$ shares the properties of the nodal set of an eigenfunction of  the Dirichlet Laplacian in $\Omega$, except that at its singular points (corresponding to meeting of more than two regular half-lines) it is permitted that an odd number of half-lines meet. We call these singular points odd.  An equi-partition has the property that all the ground state energies $\lambda(D_i)$ (i.e the first eigenvalues) of the Dirichlet Laplacian in $D_i$ are equal. Finally  the PCC property says roughly that when $D_i$ and $D_j$ are neighbours $\partial D_i \cap \partial D_j$ should be the nodal set of an eigenfunction of the Dirichlet Laplacian in the interior of  $\overline{D_i\cup D_j}$. 

Assume that $\mathcal D$ is a regular $k$-equipartition of a simply connected domain $\Omega$  that satisfies the PCC. Its energy $\Lambda(\mathcal D)$ is defined as the common value of the $\lambda(D_i)$. 

Given the $k$-partition $\mathcal D$ it is possible to define  a magnetic vector potential $\Ab^\Xb$ corresponding to it, consisting of Ahararonov--Bohm solenoids 
located at  the odd critical points $\Xb$ of the partition, and the magnetic Schrödinger  operator $T_{\Ab^\Xb}$ (as the Friedrichs extension,  originally defined on $C_0^{+\infty}(\Omega\setminus\Xb)$). With this choice of points $\Xb$, the magnetic Hamiltonian $T_{\Ab^\Xb}$ will have an eigenvalue $\mathfrak l_k$ that equals $\Lambda(\mathcal D)$.

We introduce the defect $\Def(\mathcal D)$ of the partition $\mathcal D$ as
\[
\Def(\mathcal D):=\ell(\mathcal D)-k(\mathcal D),
\]
where $\ell(\mathcal D)$ denotes the minimal labelling of the eigenvalue $\mathfrak l_k$ of the AB Hamiltonian $T_{\Ab^\Xb}$, and $k=k(\mathcal D)$ is the number of components of the partition~$\mathcal D$. It is the deficiency $\Def(\mathcal D)$ that will replace the nodal deficiency in the left-hand side in Theorem~\ref{thm:BeCoMa}.

For the right-hand side, $\dim\ker(T_{\Ab^\Xb}-\mathfrak l_k)$ will replace $\dim\ker(\laplace_\Omega-\lambda_*)$. The number of negative eigenvalues of the $\Gamma$-Dirichlet-to-Neumann operator, $\Mor\bigl(\DN(\Gamma,\lambda_*+\epsilon)\bigr)$, will be replaced by a similar term $\Mor\bigl(\DN_{\mathcal D,\Ab^\Xb}(\mathfrak l_k+\epsilon)\bigr)$. Here the operator is a magnetic Dirichlet-to-Neumann type operator, that is defined by replacing in~\eqref{eq:defDN} the Laplacian $-\Delta$ by $T_{\Ab^\Xb}$, and we refer to Subsection~\ref{sec:magneticDN} for the details.

We are thus ready to state our main result, and for simplicity we do it for simply connected domains $\Omega$.
 
\begin{theorem}\label{thm:main}
Let $\mathcal D$ be a regular $k$-equipartition of  a simply connected domain~$\Omega$ satisfying the PCC with energy $\mathfrak l_k=\Lambda(\Omega)$. 
Let $\Ab^{\Xb}$ be the associated Aharonov--Bohm potential.
Then, for sufficiently small $\epsilon >0$, 
\[
\Def(\mathcal D)= 1 - \dim \ker (T_{\Ab^\Xb} -\mathfrak l_k)+ \Mor\bigl(\DN_{\mathcal D,\Ab^\Xb}(\mathfrak l_k+\epsilon)\bigr).
\]
\end{theorem}

We willl give in Theorem~\ref{thm:flowa} a non-magnetic version of this theorem, based on certain cutting constructions of the domain.

\subsection{Outline} 
The rest of this article has the following structure. In Section~\ref{sec:equipartitions} we introduce more carefully the different types of partitions, with a focus on equipartitions. We discuss in Section~\ref{sec:AB} the Aharonov--Bohm operators and their connection with the partitions. We define the magnetic Dirichlet-to-Neumann operator in Section~\ref{sec:magneticDN}, and give the proof of Theorem~\ref{thm:main} in Section~\ref{sec:proofofmainthm}. We end Section~\ref{sec:AB} by shortly discussing another approach via double coverings. In Section~\ref{sec:cutting} we discuss the cutting construction, avoiding the Aharonov--Bohm operators. In Appendix~\ref{ex:circle} we look at a hopefully instructive toy example on a circle.

\section{Equipartitions: Notation and definitions}\label{sec:equipartitions}

In this section, we describe in which framework we will generalize the results of \cite{BeCoMa}.

\subsection{Equipartitions, nodal partitions, and minimal partitions} 
We consider a bounded connected open set $\Omega$ in $\mathbb R^2$. A \emph{$k$-partition} of $\Omega$ is a family $\mathcal D=\{D_i\}_{i=1}^k$ of mutually disjoint, connected, open sets in $\Omega$ such that $\overline\Omega=\overline{\cup_{i=1}^k D_i}$. We denote by $\mathfrak O_k(\Omega)$ the set of $k$-partitions of $\Omega$. If $\mathcal D=\{D_i\}_{i=1}^k\in \mathfrak O_k(\Omega)$ and the eigenvalues $\lambda_1(D_i)$ of the Dirichlet Laplacian in $D_i$ are equal for $1\leq i\leq k$, we say that the partition $\mathcal D$ is a \emph{spectral equipartition}. This is the type of partitions we will work on. We give two examples of how such partitions occur.

We denote by $\{\lambda_j(\Omega)\}_{j=1}^{+\infty}$ the increasing sequence of eigenvalues of the Dirichlet Laplacian in $\Omega$ and by $\{u_j\}_{j=1}^{+\infty}$ some associated orthonormal basis of real-valued eigenfunctions. The ground state  $u_1$ can be chosen to be strictly positive in $\Omega$, but the other eigenfunctions $\{u_j\}_{j\geq 2}$ must have zero sets. For a function $u\in C^0(\overline\Omega)$,  we define the \emph{zero set} $N(u)$ of $u$ as
\[
N(u)=\overline{\{x\in \Omega\:\big|\: u(x)=0\}},
\]
and call the components of $\Omega\setminus N (u)$ the \emph{nodal domains} of $ u$. Such a partition of $\Omega$ is called a \emph{nodal partition}, and we denote the number of nodal domains of $u$ by $\mu(u)$. These $\mu(u)$ nodal domains define a $k$-partition of $\Omega$, with $k=\mu(u)$. 

Since an eigenfunction $u_j$, restricted to each nodal domain satisfy the  eigenvalue equation $-\Delta u_j=\lambda_j u_j$ together with the Dirichlet boundary condition, it follows that each nodal partition is indeed a spectral equipartition. By the Courant nodal theorem, $\mu(u_j)\leq j$. We also say that the pair $(\lambda_j,u_j)$ is \emph{Courant sharp} if $\mu(u_j)=j$.

For any integer $ k\ge 1$, and for $ \mathcal D$ in $ \mathfrak O_k(\Omega)$, we introduce the \emph{energy $\Lambda(\mathcal D)$ of the partition} $\mathcal D$,
\[
\Lambda(\mathcal D)=\max_{i}\lambda_1(D_i).
\]
Then we define
\[
\mathfrak L_{k}(\Omega)=\inf_{\mathcal D\in \mathfrak O_k}\:\Lambda(\mathcal D).
\]
and  call  $ \mathcal D\in  \mathfrak O_k$ a \emph{minimal spectral $k$-partition}  if $\mathfrak L_{k}(\Omega)=\Lambda(\mathcal D)$.

If $ k=2$, it is rather well known (see~\cite{HH:2005a} or~\cite{CTV:2005}) that $ \mathfrak L_2(\Omega) =\lambda_2(\Omega)$ and that the associated minimal $2$-partition is a nodal partition, consisting of the nodal domains of some eigenfunction corresponding to second eigenvalue $\lambda_2(\Omega)$. In general, every minimal spectral partition is an equipartition (see~\cite{HHOT}).

\subsection{Regularity assumptions on partitions}\label{sec:regularityassumptions}
Attached to a  partition $\mathcal D$, we  associate a closed set in $ \overline{\Omega}$, which is called the \emph{boundary set}  of the partition:
\[
\mathcal N(\mathcal D)= \overline{ \cup_i \left( \partial D_i \cap \Omega \right)}.
\]
The set $\mathcal N(\mathcal D)$ plays the role of the nodal set (in the case of a nodal partition).

Further, we call a partition  $\mathcal D$ \emph{regular} if its associated  boundary set  $\mathcal N(\mathcal D)$ is a regular closed set in $\overline{\Omega}$. In general, a closed set $K\subset\overline{\Omega}$ is said to be \emph{regular closed} in $\overline{\Omega}$ if
\begin{enumerate}[(i)]
\item Except for finitely many distict critical points $\{x_\ell\}\subset K\cap\Omega$, the set $K$ is locally diffeomorphic to a regular curve. In the neighborhood of each critical point $x_\ell$ the set $K$ consists of a union of $\nu_\ell\geq 3$ smooth half-curves with one end at $x_\ell$.
\item The set $K\cap\partial\Omega$ consists of a (possibly empty) finite sets of boundary points $\{z_m\}$. Moreover, in a neighborhood of each  boundary point $z_m$, the set $K$ is a union of $\rho_m$ distinct smooth half-curves with one end at~$z_m$.
\item The set $K$ has the \emph{equal angle meeting property}. By this we mean that the half-curves meet with equal angle at each critical point of $K$, as well as at the boundary (together with the tangent to the boundary).
\end{enumerate}

Nodal sets are regular~\cite{Be} and in~\cite{HHOT} it is proven that minimal partitions are regular (modulo a set of capacity $0$).

For our discussion we need a weaker version of regularity which is only expressed on the \enquote{boundary set}. The first and second items remain as in the previous definition, but (iii) is changed. Indeed, we say that the closed set $K\subset\overline{\Omega}$ is \emph{weakly regular} if (i) and (ii) above hold, and further if

\begin{enumerate}[(i),resume]
\item The set $K$ has the \emph{transversal meeting property}. By this we mean the following: The set $K\cap \partial\Omega$ consists of a (possibly empty) finite set of boundary points $\{z_m\}$. Moreover $K$ is near each boundary point $z_m$ the union of $\rho_m$  smooth half-curves (with distinct tangent vectors at $z_m$) which hit $z_m$ transversally to the boundary $\partial\Omega$. Finally, at each critical point of $K$ in the interior of $\Omega$, the half-curves meet in a transversal way (i.e. no cusps).
\end{enumerate}

\begin{figure}[ht]
  \centering
  \makebox[\textwidth][c]{%
  \subcaptionbox{\label{partitions-a}}{\includegraphics[page=13]{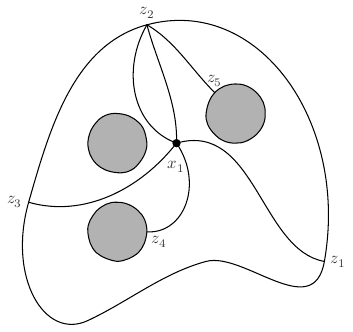}}
  \hskip 0.25cm
  \subcaptionbox{\label{partitions-b}}{\includegraphics[page=14]{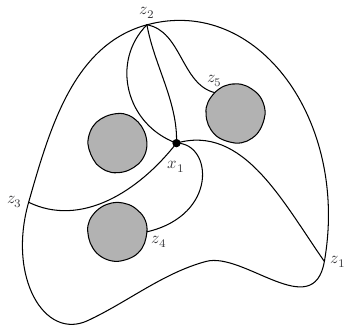}}
  }
  \caption{Partitions of a set $\Omega$ with three holes. In both cases $\nu_1=5$, $\rho_1=\rho_3=\rho_4=\rho_5=1$ and $\rho_2=3$. \subref{partitions-a} A regular partition. Note that the angles between the curves meeting at $x_1$ are $2\pi/5$ and that the angles between the curves meeting at $z_2$ and the boundary is $\pi/4$. At $z_1$, $z_3$, $z_4$ and $z_5$ the curves meet the boundary under a right angle. \subref{partitions-b} This partition is weakly regular. The curves meet the boundary and the critical point transversally, but not necessarily under equal angles.}
  \label{fig:regularpartition}
\end{figure}

\subsection{Odd and even points}\label{sec:oddevenpoints} 
Given a partition $\mathcal D$ of $\Omega$, we denote by $X^{\mathrm{odd}}(\mathcal D)$ the set of odd critical points, i.e. points $x_\ell$ for which $\nu_\ell$ is odd. Note that this set is empty when we are considering nodal partitions. When $\partial \Omega$ has one exterior boundary and $m$ interior boundaries (corresponding to $m$ holes), we should also consider the property (see~\cite{HHOO}) that an odd number of lines arrives at some component of the interior boundary (think of the hole as a point). It seems that  the assumption that there was only one boundary component was implicitly done in the litterature, or at least we should distinguish between the odd interior boundaries and the even boundaries. This would play a role in the definition of the Aharonov--Bohm operator or in the construction of the double covering.

We define by $\partial \Omega^{\mathrm{odd}}(\mathcal D)$ the union of the interior components of $\partial \Omega$ for which an odd number of lines of $\mathcal N(\mathcal D)$ arrive. In other words, we will speak of odd holes when we are in this case and $\partial \Omega^{\mathrm{odd}}(\mathcal D)$ corresponds to the union of the boundaries of the odd holes. In Figure~\ref{fig:regularpartitionoddpoints} we have marked $\partial \Omega^{\mathrm{odd}}(\mathcal D)$ in bold.

\begin{figure}[ht]
  \centering
  \subcaptionbox{\label{partitions-odd}}{\includegraphics[page=15]{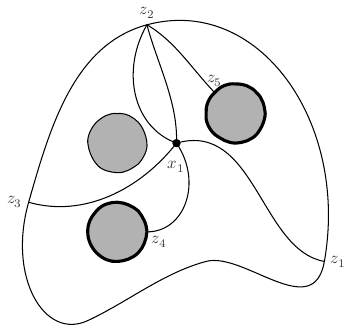}}
  \hskip 1cm
  \subcaptionbox{\label{partitions-graph}}{\includegraphics[page=16]{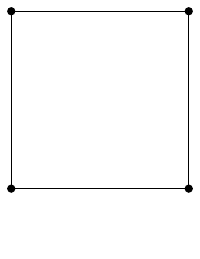}}
  \caption{\subref{partitions-odd} The partition $\mathcal D$ of $\Omega$ 
  from Figure~\ref{fig:regularpartition}\subref{partitions-a}, here with the set 
  $\partial \Omega^{\mathrm{odd}}(\mathcal D)$ in bold. 
  \subref{partitions-graph} The graph $G(\mathcal D)$ associated with the 
  partition $\mathcal D$. Note that it is bipartite.}
  \label{fig:regularpartitionoddpoints}
\end{figure}

\subsection{Pair compatibility condition}\label{sec:PCC}
Given an partition $\mathcal D=\{D_i\}$ of $\Omega$, we say that $D_i$ and $D_j$  are \emph{neighbors}, which we write $D_i\sim D_j$, if the set $D_{ij}:=\Int(\overline {D_i\cup D_j})\setminus \partial \Omega$ is connected. We associate with $\mathcal D$ a graph $G(\mathcal D)$ by associating with each $D_i$ a vertex and to each pair $ D_i\sim D_j$ an edge. We recall that a graph is said to be \emph{bipartite} if its vertices can be colored by two colors so that all pairs of neighbors have different colors. We say that $\mathcal D$ is \emph{admissible} if the associated graph $G(\mathcal D)$ is bipartite. Nodal partitions are always admissible, since the eigenfunction changes sign when going from one nodal domain to a neighbor nodal domain.

We turn to a compatibility condition between neighbors in a partition, developed in \cite{HH:2005a}. Let $\mathcal D=\{D_i\}_{i=1}^k$ be a regular equipartition of energy $\Lambda(\mathcal D)$. Given two neighbors $D_i$ and $D_j$, $\Lambda(\mathcal D)$ is the groundstate energy of both $\laplace_{D_i}$ and $\laplace_{D_j}$. There is, however, in general no way to construct a function $u_{ij}$ in the domain of $\laplace_{D_{ij}}$ such that $u_{ij} = c_i u_i$ in $D_i$ and $u_{ij} = c_j u_j$ in $D_j$. For this to be possible, it must hold that the normal derivatives of $u_i$ and $u_j$ are proportional on $\partial D_i \cap \partial D_j$.

We say that the regular partition $\mathcal D=\{D_i\}_{i=1}^k$ satisfies the \emph{pair compatibility condition}, (for short PCC), if, for some $\lambda\in \mathbb R$, and for any pair $(i,j)$ such that $D_i\sim D_j$, there is an eigenfunction $u_{ij}\not\equiv 0$ of $\laplace_{D_{ij}}$ such that $\laplace_{D_{ij}}u_{ij}=\lambda u_{ij}$, and where the nodal set of $u_{ij}$ is given by $\partial D_i\cap \partial D_j$. We refer to Figure~\ref{fig:5pcc} for some $5$-partitions of the square that satisfy the PCC. Nodal partitions and spectral minimal partitions satisfy the PCC.

\begin{remark}
In the case of bipartite equipartitions which are \enquote{generic} (i.e.  whose boundary set has no critical points and satisfies transversality conditions at the boundary) necessary and sufficient conditions to have (PCC) are given by Berkolaiko--Kuchment--Smilansky in~\cite{BKS}. They also give a formula for the nodal deficiency using the Morse index of some functional.
\end{remark}

\subsection{Admissible $k$-partitions and Courant sharp eigenvalues}
It has been proved by Conti--Terracini--Verzini~\cite{CTV0, CTV2,CTV:2005} and  Helffer--T.~Hoffmann-Ostenhof--Terracini~\cite{HHOT}, that, for any $k\in\mathbb N$, there exists a  minimal  regular $k$-partition. Other proofs of a  somewhat weaker version of this statement have been given by Bucur--Buttazzo--Henrot \cite{BBH}, Caffarelli--F.H. Lin~\cite{CL1}. It is also proven (see~\cite{HH:2005a,HHOT}) that if the graph of a minimal partition is bipartite,  then  this partition is nodal.A natural question was to determine how general the previous situation is. Surprisingly this only occurs in the Courant sharp situation.

For any integer $ k\ge 1$, we denote by $L_k(\Omega)$ the smallest eigenvalue of $\laplace_{\Omega}$, whose eigenspace contains an eigenfunction with $k$ nodal domains. We set $L_k(\Omega) = +\infty$,  if there are no  eigenfunction with $k$ nodal domains. In general, one can show that (see~\cite[Corollary~5.6]{HHOT})
\[
\lambda_k(\Omega) \leq\mathfrak L_k(\Omega) \leq L_k(\Omega) \;.
\]
The following result gives the full picture of the equality cases:

\begin{theorem}[{\cite[Theorem~1.17]{HHOT}}]\label{L=L}
Suppose that $\Omega\subset \mathbb R^2$ is smooth and that $k\in\mathbb N$. If $\mathfrak L_k(\Omega)=L_k(\Omega)$ or $\mathfrak L_k(\Omega)=\lambda_k(\Omega)$ then 
\[
\lambda_k(\Omega)=\mathfrak L_k(\Omega)=L_k(\Omega),
\]
and one can find a Courant sharp eigenpair $(\lambda_k,u_k)$.
\end{theorem}

\section{The Aharonov--Bohm approach}\label{sec:AB}

\subsection{The Aharonov--Bohm operator}
Let $\Omega\subset\mathbb R^2$ be a bounded connected domain. We recall some definitions and results about the Aharonov--Bohm (AB) Hamiltonian with poles at a finite number of points. These results were initially  motivated by the work of  Berger--Rubinstein~\cite{BeRu} and further developed in  \cite{AFT, HHOO,BHHO,BH}. We begin with the case of one pole.

\subsubsection{Simply connected $\Omega$, one AB pole}
We assume that there is one AB pole $X$ is located at $X=(x_{0},y_{0}) \in \Omega$ and introduce the magnetic vector potential
\[
{\mathbf{A}^X}(x,y) = (A_1^X(x,y),A_2^X(x,y))=\frac{\Phi}{2\pi}  \, \Bigl( -\frac{y-y_{0}}{r^2}, \frac{x-x_{0}}{r^2}\Bigr).
\]
Here $\Phi$ is the intensity of the AB magnetic field,  $\curl \mathbf{A}^X = \Phi\delta_{(x_0,y_0)}$ in the sense of distributions, and $r$ denotes the Euclidean distance between $(x,y)$ and $(x_0,y_0)$. In other words, the magnetic field vanishes identically in the punctured domain
\[
\dot\Omega_{X}= \Omega \setminus \{X\}\,,
\]
but there is a flux effect due to the non trivial circulation of the magnetic potential along a closed path enclosing $(x_0,y_0)$. We introduce the magnetic gradient $\nabla_{\Ab^X}$ as $\nabla_{\Ab^X}=\nabla-\i\Ab^X$, and consider the self-adjoint AB Hamiltonian $T_{\Ab^X}=-(\nabla_{\Ab^X})^2$. This operator is defined as the Friedrichs extension associated with the quadratic form
\[
C_0^{+\infty}(\dot\Omega_X)\ni u\mapsto \int_{\Omega}\bigl|\nabla_{\Ab^X}\,u\bigr|^2\,dx.
\]
We introduce next the multi-valued complex argument function
\[
\phi_X(x,y)=\arg\bigl(x-x_0+\i(y-y_0)\bigr).
\]
This function satisfies
\[
\Ab_X=\frac{\Phi}{2\pi}\nabla\phi_X.
\]
Thus, with the flux condition
\[
\frac{\Phi}{2\pi}=\frac 12
\]
one has
\[
-\Ab^X = \Ab^X - \nabla  \phi_{X},
\]
and that multiplication with the function $e^{\i\phi_X}$, uni-valued in $\dot\Omega_X$, is a gauge transformation intertwining $T_{\Ab^X}$ and $T_{-\Ab^X}$.

The anti-linear operator $K_X\colon L^2(\Omega)\to L^2(\Omega)$, defined by
\[
u\mapsto K_Xu=\exp(\i\phi_X)\bar u
\]
becomes a conjugation operator. In particular $(K_X)^2$ is the identity operator, $\langle K_X u,K_X v\rangle=\langle u,v\rangle$, and 
\[
K_X T_{\Ab^X}= T_{\Ab^X} K_X.
\]
We say that a function $u$ is  $K_X$-real, if it satisfies $K_{X} u =u.$ Then the operator $T_{\Ab^X}$ is preserving the $K_{X}$- real functions. In the same way  one proves that the Dirichlet Laplacian admits an orthonormal  basis of real valued eigenfunctions or one restricts this Laplacian to the vector space over $\mathbb R$ of the real-valued $L^2$ functions,  one  can construct for  $T_{{\bf A}^X}$  a basis of $K_{X}$-real eigenfunctions or, alternately, consider the restriction  of the AB Hamiltonian to the vector space over $\mathbb R$
\[
L^2_{K_{X}}(\dot{\Omega}_{X})=\{u\in L^2(\dot{\Omega}_{X}) \;,\; K_{X}\,u =u\,\}\,.
\]

\subsubsection{Simply connected $\Omega$, several AB poles}
We can extend our construction of an Aharonov--Bohm Hamiltonian in the case of a configuration with~$\ell$ distinct points $\Xb=\{X_j\}_{j=1}^\ell$ in $\Omega$  (putting a flux $\Phi=\pi$ at each of these points). We can just take as magnetic potential 
\[
\Ab^{\Xb} = \sum_{j=1}^\ell \Ab^{X_j}.
\]
The corresponding AB Hamiltonian $T_{\Ab^{\Xb}}$ is again defined as the Friedrichs extension, this time via the natural quadratic form in $C_0^{+\infty}(\dot\Omega_{\Xb})$, where $\dot{\Omega}_{\Xb} = \Omega \setminus \Xb$.

We can also construct (see \cite{HHOO}) the anti-linear operator $K_{\bf X}$,  where $\phi_{X}$ is replaced by a multivalued function
\[
\Phi_\Xb=\sum_{j=1}^\ell \phi_{X_j}
\] 
which satisfies $\nabla \Phi_\Xb = 2 \Ab^{\Xb}$. Moreover 
\[
\exp(\i\Phi_\Xb)= \prod_{j=1}^\ell  \exp(\i \phi_{X_j})
\] 
is uni-valued and belongs to  $C^\infty(\dot \Omega_{\Xb})$. As in the case of one AB pole, we can consider the (real) subspace of the $K_{\Xb}$-real functions in $L^2_{K_{\Xb}}(\dot{\Omega}_{\Xb})$, and our operator as an unbounded selfadjoint operator in $L^2_{K_{\Xb}}(\dot{\Omega}_{\Xb})$.

\subsubsection{Non-simply connected $\Omega$}
If $\Omega$ is not simply connected, we also accept some of the Aharonov--Bohm fluxes to be placed in holes in the bounded components of the complement of $\Omega$. If $\Omega$ has $m$ holes $\omega_1, \ldots, \omega_m$,  we allow $m'$ \enquote{odd} poles $X'_i$  (with $0\leq m' \leq m$ and $X'_i \in \omega_i$ for $i\in \{1,\dots,m'\}$).

We call \emph{odd holes} the holes for which we have introduced the poles $X'_j$.  In this case, we can reproduce the same construction. We call the remaining holes in $\Omega$ \emph{even}.

At the end, for $\widehat {\bf X} := \Xb\cup\Xb'$, we have constructed an AB Hamiltonian in  $\dot{\Omega}_{\Xb}$ associated with a magnetic potential $\Ab^{\widehat \Xb}$ with poles at $\Xb$ and half renormalized flux created by the magnetic potential in each odd hole, the flux created in the even hole being $0$. Similarly, we can as before use the \enquote{conjugation} operator $K_{\widehat{\Xb}}$.

It was shown in~\cite{HHOO,AFT} (inspired by the previous work  \cite{BeRu}) that the nodal set of such a $K_{\widehat \Xb}$-real eigenfunction has the same structure as the nodal set of a real-valued  eigenfunction of the Dirichlet Laplacian except that an odd number of half-lines meet at each pole (see Subsection~\ref{sec:oddevenpoints}), and that the number of lines meeting the interior boundary should be odd (resp. even) at the boundary of an odd (resp. even) hole.

\subsection{Equipartitions and nodal partitions of AB Hamiltonians}\label{Section5}
We start from constructions introduced in~\cite{HH:2005a,HH:2015b,BH}. Suppose $\Omega$ is a bounded, simply connected (thus, to simplify, we describe the case without holes),  domain and that $\partial\Omega$ is piecewise differentiable. Let $\mathcal D$ be a regular $k$-equipartition with energy $\Lambda(\mathcal D)=\mathfrak l_k(\Omega)$ satisfying the PCC.
  
We denote by $\Xb=X^{\mathrm{odd}}(\mathcal D)=\{X_j\}_{j=1}^\ell$ the critical points of the boundary set $\mathcal N(\mathcal D)$ of the partition for which an odd number of half-curves meet. For this family of points $\Xb$, $\mathfrak l_k(\Omega)$ is an eigenvalue of the AB Hamiltonian associated with $\dot{\Omega}_{\Xb}$, and we can explicitly construct the corresponding eigenfunction with $k$ nodal domains described by $\{D_i\}$.

In~\cite{HHOT} it was proven that there exists a family $\{u_i\}_{i=1}^k$ of functions such that $u_i$ is a ground state of $\laplace_{D_i}$ and $u_i -u_j$ is a second eigenfunction of $\laplace_{D_{ij}}$ when $D_i\sim D_j$ (here we have extended $u_i$ and $ u_j$ by $0$ outside of $D_i$ and $D_j$, respectively, and we recall that $D_{ij}=\Int(\overline{D_i\cup D_j})$).
  
The claim (see the proof of Theorem 5.1 (step 1) in~\cite{HH:2015b})  is that one can  find a sequence $\epsilon_i(x)$ of $\mathbb S^1$-valued functions, where $\epsilon_i$ is a suitable\footnote{\label{foot1}Note that by construction the $D_i$'s never contain any point of $\Xb$. By Euler formula a path $\gamma$ in $D_i$ can only contain an even number of points in $\Xb$. Hence the square root $D_i \ni x \mapsto \exp\bigl(\i\Phi_{\Xb}(x)/2\bigr)$ is well defined  up to a multiplicative constant of modulus $1$ and the ground state energy of the Dirichlet Laplacian $L_{D_i}$ is the same as the ground state energy of $H_{\Ab^{\Xb}}$ in $D_i$. See~\cite{Le}.} square root of  $\exp(\i\Phi_{\Xb})$ in $D_i$, such that $\sum_i \epsilon_i(x) u_i(x)$ is an eigenfunction of the AB Hamiltonian $T_{\Ab^{\Xb}}$ associated with the eigenvalue $\mathfrak l_k(\Omega)=\Lambda(\mathcal D)$.

\subsection{The Berkolaiko--Cox--Marzuola construction in the Aharonov--Bohm approach}
We follow the approach of \cite{BeCoMa} but to be able to treat non-admissible  partitions we introduce the AB Hamiltonian attached to the partition. Thus, let $\mathcal D$ be a $k$-partition in $\Omega$. We denote by $\Gamma=\mathcal N(\mathcal D)$ the boundary set of the partition in $\Omega$, and by $m_k$ the multiplicity of $\mathfrak l_k(\Omega)$ as eigenvalue of the magnetic AB Hamiltonian $T_{\Ab^{\Xb}}$, defined above.

We consider the family $\{\mathfrak B_\sigma\}_{\sigma \in \mathbb R}$ of sesquilinear forms defined on the magnetic Sobolev space $H_{0,\Ab^\Xb}^1(\Omega)\times H_{0,\Ab^\Xb}^1(\Omega)$ (see L\'ena~\cite{Le} and also~\cite{GoSc}) by
\[
(u,v)\mapsto \mathfrak t_\sigma (u,v) = \int_\Omega \nabla_{\Ab^\Xb}  u \cdot \overline{\nabla_{\Ab^\Xb} v} + \sigma \int_\Gamma u\, \overline{v} \,ds_\Gamma,
\]
where $\Ab^\Xb$ is the magnetic AB potential and $ds_\Gamma$ is the induced measure on (each arc of) $\Gamma$.

We should explain how the integral over $\Gamma$ is to be interpreted. For each arc $\gamma_i$ in $\Gamma\cap \dot\Omega_{\Xb}$, we can define in its neighborhood $V(\gamma_i)$ in $\dot \Omega_{\Xb}$ a $C^\infty$ square root $\exp(\i\Phi_{\Xb}/2)$ of $\exp(\i\Phi_{\Xb})$  and we have $\exp(\i\Phi_{\Xb}/2) u\in H^1(V(\gamma_i))$ if $u\in H^1_{\Ab^\Xb}(\Omega)$. We can then define
\[
\int_{\gamma_i} u\,\overline{v} \,ds_{\gamma_i} 
:= 
\int_{\gamma_i }(\exp(\i\Phi_{\Xb}/2) u) \cdot ( \overline{\exp(\i\Phi_{\Xb}/2) v})\, ds_{\gamma_i},
\]
where we use the standard trace for an element of $H^1$. Note that, with this definition, the \enquote{magnetic trace space} on $\gamma_i$ is identified as
\[
H^{1/2}_{\Ab^\Xb} (\gamma_i):= \exp(-\i\Phi_{\Xb}/2) H^{1/2} (\gamma_i).
\]
We further set $H^{1/2}_{\Ab^\Xb} (\Gamma):= \oplus_i H^{1/2}_{\Ab^\Xb} (\gamma_i)$, and writing 
\[
\int_\Gamma  u\, \overline{v} \,ds_\Gamma  
=\sum_i \int_{\gamma_i } u\, \overline{v} \,ds_{\gamma_i},
\]
this permits to show that the sesquilinear form  $\mathfrak t_\sigma$ is continuous.

Associated with this sesquilinear form we have the corresponding magnetic-Robin AB Hamiltonian $T_{\Ab^\Xb,\sigma}$  defined as the Friedrichs extension. We also define $T_{\Ab^\Xb,+\infty}$ as the corresponding AB magnetic Schr\"odinger  operator, with Dirichlet boundary conditions at $\partial\Omega\cup\Gamma$. We collect some properties of the operators $\{T_{\Ab^\Xb,\sigma}\}$.

\begin{proposition}\label{prop:transmission}
Assume that $\mathcal D$ is a weakly regular partition of $\Omega$.

The self-adjoint operators $\{T_{\Ab^\Xb,\sigma}\}$, $-\infty<\sigma\leq+\infty$, 
have compact resolvents.

Moreover, if $\sigma\in\mathbb R$, then the domain of $T_{\Ab^\Xb,\sigma}$ consists of all elements $u\in H^1_{0,\Ab^\Xb}(\Omega)$ such that $(\nabla_{\Ab^\Xb})^2u\in L^2(\Omega)$, and such that the following transmission conditions are satisfied:

If $D_i$ and $D_j$ are two neighbors in the partition $\mathcal D$ of $\Omega$, and $\gamma$ is a regular arc in $\partial D_i\cap \partial D_j$, then, on $\gamma$,
\begin{equation}
\label{eq:transmission}
\nu_i \cdot \nabla_{\Ab^\Xb} u_{i } + \nu_j \cdot \nabla_{\Ab^\Xb} u_{j} = -\sigma u\,,
\end{equation}
where $\nu_i$ and $\nu_j$ denote the exterior normals to $D_i$ and $D_j$ (at a point of $\gamma$) and $u_i$ and $u_j$ denote the restrictions of $u$ to $D_i$ and $D_j$, respectively.
\end{proposition}

\begin{proof}
The proof follows in the same manner as for the Laplace operator, with small additions or modifications. We refer to \cite[Proposition~2.2]{AM0} for the characterization of the domain, and the compactness of the resolvent. Here, one should note that the magnetic Sobolev space $H^1_{\Ab^\Xb}(\Omega)$ is continuously embedded in the ordinary Sobolev space $H^1(\Omega)$ if the fluxes around the poles are non-integers (see~\cite[Corollary~2.5]{Le}).

For the transmission conditions along the boundary set, we refer to~\cite{BeCoMa}.
\end{proof}

Given $-\infty<\sigma\leq+\infty$, we denote by $\{\hat\lambda_n(\sigma)\}_{n\in\mathbb N}$ the analytic eigenvalue branches of $L_\sigma$, and we enumerate by $\{\lambda_n(\sigma)\}_{n\in\mathbb N}$ the increasing sequence of eigenvalues of $L_\sigma$, counted with multiplicity. As in~\cite[Lemma 2]{BeCoMa}, a perturbative argument shows that $\sigma\mapsto \hat\lambda_n(\sigma)$ is either strictly increasing or equal to $\hat\lambda_n(0)$, and the latter case only occurs when $\hat\lambda_n(0)$ is an eigenvalue of $L_{+\infty}$.

\begin{proposition}
As $\sigma\to+\infty$, $\lambda_n(\sigma)\to\lambda_n(+\infty)$.
\end{proposition}

\begin{proof}
The resolvents of $T_{\Ab^\Xb,\sigma}$ converge to the resolvent of $T_{\Ab^\Xb,+\infty}$ as $\sigma\to+\infty$, see~\cite[Proposition~2.6]{AM0} for the proof in the case of the Laplacian. It then follows (see~\cite[Proposition~2.8]{AM0}) that the eigenvalues also converge. 
\end{proof}

The operator $T_{\Ab^\Xb,+\infty}$ can be identified as the direct sum of the AB magnetic Schr\"odinger operators with vector potential $\Ab^\Xb$ on each component $D_i$ of the partition $\mathcal D$, with Dirichlet boundary conditions on $\partial D_i$. On each $D_i$, the magnetic potential can be gauged away by a square root $\exp \bigl(\i \Phi_\Xb /2\bigr)$ which can be defined as a univalued function in each $D_i$  (see footnote~\ref{foot1}). Moreover $\exp \bigl(\i \Phi_\Xb /2\bigr)$ is unique up to a multiplicative constant of modulus~$1$.  We denote by $\Theta_\Xb^i$ a choice of this square root in $D_i$ and observe that $\Theta_\Xb^i$ can be extended to $\overline{D_i}$ except at the singular points of $\partial D_i$. 

From the gauge transforms above it follows that $\mathfrak l_k$ is the smallest eigenvalue of the operator $T_{\Ab^\Xb,+\infty}$.

\subsection{The magnetic Dirichlet-to-Neumann operator}\label{sec:magneticDN}
We can now construct the magnetic Dirichlet-to-Neumann operator. For this we proceed in the following way. For each $D_j$, we consider $\partial D_j \cap \Omega$. We introduce  the  magnetic Dirichlet--Neumann operator on $\partial D_j$ which associates, for  $\lambda >\mathfrak l_k$  small enough, to a  function $h\in H^{1/2}_{\Ab^\Xb}(\partial D_j)$, vanishing on $\partial \Omega \cap \partial D_j$ a solution $u$ to
\begin{equation}\label{mdn}
\begin{cases}
 -(\nabla_{\Ab^\Xb})^2  u= \lambda u & \text{in }D_j,\\
u=h & \text{on } \partial D_j\,.
\end{cases}
\end{equation}
Even though it is not our case, we start by explaining the procedure in the case that $\Ab^\Xb$ is replaced by a magnetic vector potential $\Ab$ that is regular. Then we define a pairing of elements in $H^{-1/2}_\Ab(\partial D_j)$ and $H^{1/2}_\Ab(\partial D_j)$, inspired by how it is done in the non-magnetic case by the Green--Riemann formula.

Observing that $\lambda$ is not an eigenvalue of the Dirichlet realization of $T_{\Ab}$ in $D_j$, the map $h\mapsto u$ (where $u$ is a solution of~\eqref{mdn}) is continuous from $H^{1/2}_{\Ab} (\partial D_j)$ into $H^1_\Ab ( D_j)$ and we have for some positive constant $C$
\begin{equation} \label{eq:conta}
\|u\|_{H^1_{\Ab}(D_j)} \leq C \|h\|_{H^{1/2}_{\Ab}(\partial D_j)}\,.
\end{equation}
We consider a second pair $(\tilde u,\tilde h)$ solution of \eqref{mdn} and set 
\begin{equation}
\label{eq:defpairing}
\begin{aligned}
\bigl\langle \nu_j\cdot\nabla_\Ab u,\tilde h\bigr\rangle_{H^{-1/2}_{\Ab}(\partial D_j),H^{1/2}_{\Ab}(\partial D_j)}
&:=
\langle \nabla_\Ab u,\nabla_\Ab \tilde u\rangle_{L^2(D_j)}
+
\langle (\nabla_\Ab)^2u,\tilde u\rangle_{L^2(D_j)} \\
& = 
\langle \nabla_\Ab u,\nabla_\Ab \tilde u\rangle_{L^2(D_j)}
- \lambda 
\langle u,\tilde u\rangle_{L^2(D_j)},
\end{aligned}
\end{equation}
where $\nu_j$ is  the exterior normal derivative to $\partial D_j$.

From this formula, we see, using~\eqref{eq:conta} for the pairs $(u,h)$ and $(\tilde u,\tilde h)$, that there exists a positive $C$ such that
\[
|\bigl\langle \nu_j\cdot\nabla_\Ab u,\tilde h\bigr\rangle_{H^{-1/2}_{\Ab}(\partial D_j),H^{1/2}_{\Ab}(\partial D_j)}| \leq C \| h\|_{H^{1/2}_{\Ab} (\partial D_j)}\, \|\tilde h\|_{H^{1/2}_{\Ab} (\partial D_j)}
\]
This shows the continuity of the Dirichlet to Neumann magnetic operator  from $H^{1/2}_{ \Ab} (\partial D_j)$ into its dual.

We now turn to our singular case with $\Ab=\Ab^\Xb$. Actually, to avoid possible problems with the singularities of $\Ab^\Xb$, we will see that we can transfer the situation in each $D_i$ to the non-magnetic case.

We start by looking at one of the partitions $D_i$. We assume that $\lambda\in\mathbb R$ is not in the spectrum of the Dirichlet realization of $-(\nabla_{\Ab^\Xb})^2$ in $D_i$, and that $(u,h)$ satisfies
\begin{equation}\label{mdn2}
  \begin{cases}
    -(\nabla_{\Ab^\Xb})^2u=\lambda u & \text{in $D_i$},\\
    u=h & \text{on $\partial D_i$}.
  \end{cases}
\end{equation}
Let $v=(\Theta_\Xb^i)^{-1}u$. Then, since $\nabla_{\Ab^\Xb}=\Theta_\Xb^i\nabla (\Theta_\Xb^i)^{-1}$, we find that $v$ satisfies
\[
  \begin{cases}
    -\nabla^2 v=\lambda v & \text{in $D_i$},\\
    v=(\Theta_\Xb^i)^{-1} h & \text{on $\partial D_i$}.
  \end{cases}
\]
Inspired by this, we define our magnetic trace space $H^{1/2}_{\Ab^\Xb}(\partial D_i)$ by saying that $h\in H^{1/2}_{\Ab^\Xb}(\partial D_i)$ precisely when $(\Theta_\Xb^i)^{-1} h$ belongs to $H^{1/2}(\partial D_i)$. Moreover, if this is the case, then the normal derivative $\nu_i\cdot \nabla((\Theta_\Xb^i)^{-1} u)$ is also well-defined as an element of $H^{-1/2}(\partial D_i)$, and it is natural to define on $\partial D_i$ the magnetic normal derivative of $u$ as
\[
  \nu_i \cdot \nabla_{\Ab^\Xb} u|_{\partial D_i}
    :=
    \Theta_\Xb^i\bigl(\nu_i\cdot\nabla((\Theta_\Xb^i)^{-1}\,u)\bigr).
\]
Thus, given boundary data $h\in H^{1/2}_{\Ab^\Xb}(\partial D_i)$, we define the reduced magnetic Dirichlet-to-Neumann operator $\DN_{i,\Ab^\Xb}(\lambda)$ by first identifying $u$ such that the pair $(u,h)$ satisfies~\eqref{mdn2}, and then applying the magnetic normal derivative. This provides us with the following definition
\begin{equation}\label{eq:mbd}
    \DN_{i,\Ab^\Xb}(\lambda)h
    =
    \nu_i \cdot \nabla_{\Ab^\Xb} u|_{\partial D_i}
    =
    \Theta_\Xb^i\bigl(\nu_i\cdot\nabla((\Theta_\Xb^i)^{-1}\,u)\bigr).
\end{equation}
Moreover, given another pair $(\tilde u,\tilde h)$ satisfying~\eqref{mdn2} we define the pairing of $\DN_{i,\Ab^\Xb}(\lambda)h$ with $\tilde h$ as
\[
  \begin{multlined}
  \bigl\langle \DN_{i,\Ab^\Xb}(\lambda)h,\tilde h\bigr\rangle_{H^{-1/2}_{\Ab^\Xb}(\partial D_i),H^{1/2}_{\Ab^\Xb}(\partial D_i)}\\
  :=
  \bigl\langle \nu_i\cdot \nabla( (\Theta_\Xb^i)^{-1} \, u),(\Theta_\Xb^i)^{-1}\,\tilde h\bigr\rangle_{H^{-1/2}(\partial D_i),H^{1/2}(\partial D_i)}.
  \end{multlined}
\]
We are now ready to define the magnetic Dirichlet-to-Neumann operator $\DN_{\mathcal D,\Ab^{\Xb}}(\lambda)$. We recall that $R_{\Gamma\shortto \partial D_i}$ denotes the operator that first extends by $0$ to $\Gamma\cup\partial\Omega$ and then restricts to $\partial D_i$, and that $R_{\partial D_i\shortto\Gamma}$ is the operator that first extends an element defined on $\partial D_i$ by $0$ so that it becomes defined on $\Gamma\cup\partial\Omega$, and then restricts the result to $\Gamma$.

We define the space $H^{1/2}_{\Ab^\Xb}(\Gamma)$ to be the set of $h$ defined on $\Gamma$ such that $R_{\Gamma\shortto\partial D_i}h$ belongs to $H^{1/2}_{\Ab^\Xb}(\partial D_i)$ for each $i$, $1\leq i\leq k$. Applying $\DN_{i,\Ab^\Xb}(\lambda)$ to $R_{\Gamma\shortto\partial D_i}h$, we end up with an element in $H^{-1/2}_{\Ab^\Xb}(\partial D_i)$. If $\tilde h$ is another element of $H^{1/2}_{\Ab^\Xb}(\Gamma)$, we can pair the result obtained with $R_{\Gamma\shortto\partial D_i}\tilde h$,
\[
  \bigl\langle \DN_{i,\Ab^\Xb}(\lambda)R_{\Gamma\shortto\partial D_i}h,R_{\Gamma\shortto\partial D_i}\tilde h\bigr\rangle_{H^{-1/2}_{\Ab^\Xb}(\partial D_i),H^{1/2}_{\Ab^\Xb}(\partial D_i)}
\]
We next define
\[
  \DN_{\mathcal D,\Ab^{\Xb}}(\lambda)\colon H^{1/2}_{\Ab^\Xb} (\Gamma)\to H^{-1/2}_{\Ab^\Xb} (\Gamma).
\]
as the sum
\begin{equation}\label{eq:4.5}
\DN_{\mathcal D,\Ab^{\Xb}}(\lambda)
=
\sum_{i=1}^k R_{\partial D_i\shortto\Gamma} \, \DN_{i,\Ab^\Xb}(\lambda)\, R_{\Gamma\shortto \partial D_i}.
\end{equation}
Defining the pairing between $H^{1/2}_{\Ab^\Xb}(\Gamma)$ and $H^{-1/2}_{\Ab^\Xb}(\Gamma)$ by summing over $i$ the pairing relative to each $D_i$, we get  for the Dirichlet-to-Neumann operator:
\[
\begin{multlined}
  \bigl\langle\DN_{\mathcal D,\Ab^{\Xb}}(\lambda)h,\tilde h\big\rangle_{H^{-1/2}_{\Ab^\Xb}(\Gamma),H^{1/2}_{\Ab^\Xb}(\Gamma)}\\
  : =   \sum_{i=1}^k\bigl\langle \DN_{i,\Ab^\Xb}(\lambda)R_{\Gamma\shortto\partial D_i}h,R_{\Gamma\shortto\partial D_i}\tilde h\bigr\rangle_{H^{-1/2}_{\Ab^\Xb}(\partial D_i),H^{1/2}_{\Ab^\Xb}(\partial D_i)},\\
  \forall h, \tilde h \mbox{ in } H^{1/2}_{\Ab^\Xb}(\Gamma)\,.
  \end{multlined}
\]

\begin{proposition}\label{prop:DNsa}
The operator $\DN_{\mathcal D,\Ab^{\Xb}}(\lambda)$ is self-adjoint.
\end{proposition}

\begin{remark}
For non-singular magnetic potentials, one usually bases the proof of self-adjointness on the magnetic Green--Riemann formula, i.e. the first equality in~\eqref{eq:defpairing}. For our singular potentials, since we define our magnetic Neumann-to-Dirichlet operator in terms of the non-magnetic one, the statement follows by the corresponding statement for the non-magnetic Dirichlet-to-Neumann operator, relying on the Green--Riemann formula.
\end{remark}

\begin{proof}
Let $h$ and $\tilde h$ be in $H^{1/2}_{\Ab^\Xb}(\Gamma)$. By the construction, it is sufficient to show that, for all $i$ with $1\leq i\leq k$,
\begin{equation}\label{eq:DNsa}
  \begin{multlined}
    \bigl\langle \DN_{i,\Ab^\Xb}(\lambda)R_{\Gamma\shortto\partial D_i}h,R_{\Gamma\shortto\partial D_i}\tilde h\bigr\rangle_{H^{-1/2}_{\Ab^\Xb}(\partial D_i),H^{1/2}_{\Ab^\Xb}(\partial D_i)}\\
    =
    \bigl\langle R_{\Gamma\shortto\partial D_i}h,\DN_{i,\Ab^\Xb}(\lambda)R_{\Gamma\shortto\partial D_i}\tilde h\bigr\rangle_{H^{-1/2}_{\Ab^\Xb}(\partial D_i),H^{1/2}_{\Ab^\Xb}(\partial D_i)}.
  \end{multlined}
\end{equation}
Let $u$ and $\tilde u$ be elements of $H^1_{\Ab^\Xb}(D_i)$ such that the pairs $(u,R_{\Gamma\shortto\partial D_i}h)$ and $(\tilde u,R_{\Gamma\shortto\partial D_i}\tilde h)$ satisfy~\eqref{mdn2}. The left-hand side in~\eqref{eq:DNsa} is by definition equal to
\[
\bigl\langle \nu_j\cdot \nabla( (\Theta_\Xb^i)^{-1} \, u), (\Theta_\Xb^i)^{-1}\,R_{\Gamma\shortto\partial D_i}\tilde h\bigr\rangle_{H^{-1/2}(\partial D_i),H^{1/2}(\partial D_i)},
\]
which by the Green--Riemann formula for the Laplace operator equals
\[
\bigl\langle \nabla ((\Theta_\Xb^i)^{-1}\, u),\nabla ((\Theta_\Xb^i)^{-1}\, \tilde u)\bigr\rangle_{L^2(\partial D_i)}
+
\bigl\langle \nabla^2 ((\Theta_\Xb^i)^{-1}\, u),(\Theta_\Xb^i)^{-1}\, \tilde u\rangle_{L^2(\partial D_i)}.
\]
Next, we use that $\nabla_{\Ab^\Xb}=\Theta_\Xb^i\nabla (\Theta_\Xb^i)^{-1}$ and that $\Theta_\Xb^i$ has modulus one, to write this as
\[
\bigl\langle \nabla_{\Ab^\Xb}u,\nabla_{\Ab^\Xb}\tilde u\bigr\rangle_{L^2(\partial D_i)}
+
\bigl\langle (\nabla_{\Ab^\Xb})^2  u), \tilde u\rangle_{L^2(\partial D_i)}.
\]
Since $(\nabla_{\Ab^\Xb})^2u=-\lambda u$ in $D_i$, we end up with
\[
\bigl\langle \nabla_{\Ab^\Xb}u,\nabla_{\Ab^\Xb}\tilde u\bigr\rangle_{L^2(\partial D_i)}
-
\lambda\bigl\langle u, \tilde u\rangle_{L^2(\partial D_i)}.
\]
This expression is symmetric in $u$ and $\tilde u$, so we will obtain it again if we start with the right-hand side of~\eqref{eq:DNsa} and exchange the roles of $(u,R_{\Gamma\shortto\partial D_i}h)$ and $(\tilde u,R_{\Gamma\shortto\partial D_i}\tilde h)$.
\end{proof}

\subsection{Proof of Theorem~\ref{thm:main}}\label{sec:proofofmainthm}
Following~\cite{BeCoMa}, we denote by $\Mor\bigl(\DN_{\mathcal D,\Ab^{\Xb}}(\lambda)\bigr)$ the number of negative eigenvalues of $\DN_{\mathcal D,\Ab^{\Xb}}(\lambda)$, including counting multiplicities. We introduce the defect $\Def(\mathcal D)$ of the partition $\mathcal D$ as
\[
\Def(\mathcal D):=\ell(\mathcal D)-k(\mathcal D),
\]
where $\ell(\mathcal D)$ denotes the minimal labelling of the eigenvalue $\mathfrak l_k$ of the AB Hamiltonian $T_{\Ab^\Xb}$, and $k=k(\mathcal D)$ is the number of components of the partition~$\mathcal D$. We are ready to state our main result, and for simplicity we do it for simply connected domains $\Omega$ (the only change for the general case would be in the definition of the magnetic Aharonov--Bohm potential). We restate Theorem~\ref{thm:main}:

\begingroup
\def\thetheorem{\ref{thm:main}}
\begin{theorem}
  Let $\mathcal D$ be a regular $k$-equipartition of  a simply connected domain~$\Omega$ satisfying the PCC with energy $\mathfrak l_k=\mathfrak l_k(\Omega)$. 
  Let $\Ab^{\Xb}$ be the associated Aharonov--Bohm potential.
  Then, for sufficiently small $\epsilon >0$, 
  \[
  \Def(\mathcal D)= 1 - \dim \ker (T_{\Ab^\Xb} -\mathfrak l_k)+ \Mor\bigl(\DN_{\mathcal D,\Ab^\Xb}(\mathfrak l_k+\epsilon)\bigr).
  \]
  \end{theorem}
\addtocounter{theorem}{-1}
\endgroup

The key ingredient in the proof is the connection between the negative eigenvalues of $\DN_{\mathcal D,\Ab^\Xb}(\mathfrak l_k+\epsilon)$ and eigenvalues of the family $L_\sigma$, given in the following lemma. 

\begin{lemma}\label{lem:eigeig}
Assume that $\sigma>0$ and that $\epsilon>0$ is sufficiently small. Then $-\sigma$ is an eigenvalue of $\DN_{\mathcal D,\Ab^\Xb}(\mathfrak l_k+\epsilon)$ if, and only if, $\mathfrak l_k+\epsilon$ is an eigenvalue of $L_\sigma$. If this is the case, then the multiplicities agree.
\end{lemma}

We refer the reader to~\cite[Theorem~4.1]{AM0} and to~\cite[Lemma~1]{BeCoMa} for proofs of similar statements, but give the details below for completeness.

\begin{proof}
Assume first that $-\sigma$ is a negative eigenvalue of $\DN_{\mathcal D,\Ab^\Xb}(\mathfrak l_k+\epsilon)$ with eigenfunction $h\in H^{1/2}_{\Ab^\Xb}(\Gamma)$. Let $D_i$ and $D_j$ be two neighbors in the partition $\mathcal D$. Then, on their common boundary $\partial D_i\cap \partial D_j$ in $\Omega$,
\[ 
  \bigl(\DN_{i,\Ab^\Xb}(\mathfrak l_k+\epsilon) R_{\Gamma\shortto\partial D_i}+
  \DN_{j,\Ab^\Xb}(\mathfrak l_k+\epsilon) R_{\Gamma\shortto\partial D_j}\bigr)h=-\sigma h.
\]
We denote by $u_i\in H^1_{\Ab^\Xb}(D_i)$ a function that satisfies
\begin{equation}\label{eq:eigLsigma}
\begin{cases}
  -(\nabla_{\Ab^\Xb})^2u_i=(\mathfrak l_k+\epsilon)u_i,&\text{in $D_i$},\\
  u_i=R_{\Gamma\shortto\partial D_i}h,&\text{on $\partial D_i$},
\end{cases}
\end{equation}
and we define $u_j$ similarily, with $i$ replaced by $j$. As in~\eqref{eq:mbd} we interpret the magnetic Dirichlet-to-Neumann operator as magnetic boundary derivatives,
\begin{equation}\label{eq:DNi}
  R_{\partial D_i\shortto\Gamma}\DN_{i,\Ab^\Xb}(\mathfrak l_k+\epsilon) {R_{\Gamma\shortto\partial D_i}h}  =
  \nu_i\cdot \nabla_{\Ab^\Xb}u_i,
\end{equation}
and
\begin{equation}\label{eq:DNj}
  R_{\partial D_j\shortto\Gamma}\DN_{j,\Ab^\Xb}(\mathfrak l_k+\epsilon) {R_{\Gamma\shortto\partial D_j}h}
  =
  \nu_j\cdot \nabla_{\Ab^\Xb}u_j.
\end{equation}
Thus, we find that on the common boundary,
\begin{equation}\label{eq:commonboundary}
  \nu_i\cdot \nabla_{\Ab^\Xb}\,u_i+\nu_j\cdot \nabla_{\Ab^\Xb}\,u_j=-\sigma h.
\end{equation}
Taking into account that we use an outward pointing normal, this is nothing but the transmission conditions from~\eqref{eq:transmission} in Proposition~\ref{prop:transmission}. This means that $u$, defined by $u_i$ in each $D_i$ and as $h$ on $\Gamma$ and as $0$ on $\partial\Omega$ actually belongs to the domain of $T_{\Ab^\Xb,\sigma}$ and from~\eqref{eq:eigLsigma} we also find that $u$ is an eigenfunction of $T_{\Ab^\Xb,\sigma}$ with eigenvalue $\mathfrak l_k+\epsilon$.

Next, assume that $u$ is an eigenfunction of $T_{\Ab^\Xb,\sigma}$ corresponding to the eigenvalue $\mathfrak l_k+\epsilon$. We recall that the trace of $u$ at $\partial\Omega$ equals $0$. Let $h$ denote the restriction of $u$ to $\Gamma$. Then $u$ satisfies the transmission conditions in~\eqref{eq:transmission}. This means that, if we denote by $u_i$ and $u_j$ the restrictions of $u$ to two neighbors $D_i$ and $D_j$, then~\eqref{eq:commonboundary} is satisfied on their common boundary $\partial D_i\cap\partial D_j$. Now, these magnetic Neumann derivatives agree with the magnetic Dirichlet-to-Neumann operators we defined, so~\eqref{eq:DNi} and~\eqref{eq:DNj} hold. But this means precisely that $h$ is an eigenfunction of $\DN_{\mathcal D,\Ab^\Xb}(\mathfrak l_k+\epsilon) h$ with eigenvalue $-\sigma$.

To see that the multiplicites of the eigenspaces agree, one notices that the mapping from $\ker(T_{\Ab^\Xb,\sigma}-(\mathfrak l_k+\epsilon))$ to $\ker(\DN_{\mathcal D,\Ab^\Xb}(\mathfrak l_k+\epsilon)+\sigma)$ that sends $u$ to the trace of $u$ at $\Gamma$ is an isomorphism. We have seen above that the mapping is well-defined. Moreover, it is injective, since, if $u\in\ker(T_{\Ab^\Xb,\sigma}-(\mathfrak l_k+\epsilon))$ and the trace of $u$ at $\Gamma$ equals zero, then $u$ belongs to $\ker(T_{\Ab^\Xb,+\infty}-(\mathfrak l_k+\epsilon))$. We know, however, that $\mathfrak l_k$ is the smallest eigenvalue of $T_{\Ab^\Xb,+\infty}$, and so with $\epsilon>0$ sufficiently small, we have that $\mathfrak l_k+\epsilon$ cannot be an eigenvalue of $T_{\Ab^\Xb,+\infty}$. Hence $u=0$.
\end{proof}

We are now ready to give the proof of Theorem~\ref{thm:main}.

\begin{proof}[Proof of Theorem~\ref{thm:main}]
The proof is similar to the proof of~\cite[equation~(3)]{BeCoMa}. The first eigenvalue $\lambda_1(+\infty)$ of $T_{\Ab^\Xb,+\infty}$ equals the first Dirichlet eigenvalue of the Laplace operator on each component of $\mathcal D$, and hence it has multiplicity $k(\mathcal D)$. Moreover, since $\lambda_1(+\infty)=\mathfrak l_k$,
\[
\lim_{\sigma\to+\infty}\lambda_n(\sigma)\,\,
\begin{cases}
=\mathfrak l_k, & 1\leq n\leq k,\\
>\mathfrak l_k, & n>k.
\end{cases}
\]
The operator $T_{\Ab^\Xb,0}=T_{\Ab^\Xb}$, on the other hand, has $\ell(\mathcal D)+\dim \ker (T_{\Ab^\Xb} -\mathfrak l_k)-1$ eigenvalues less than or equal to $\mathfrak l_k$, and exactly $k(\mathcal D)$ of them will converge to $\mathfrak l_k$ as $\sigma\to+\infty$. This means that $\ell(\mathcal D)+\dim \ker (T_{\Ab^\Xb} -\mathfrak l_k)-1-k(\mathcal D)$ eigenvalues of $T_{\Ab^\Xb,\sigma}$ will cross $\mathfrak l_k+\epsilon$ for some finite $\sigma>0$, if $\epsilon>0$ is sufficiently small.

According to Lemma~\ref{lem:eigeig}, and observing the monotonicity of $\sigma\mapsto \lambda_n(\sigma)$,  every such crossing gives rise to a negative eigenvalue of $\DN_{\mathcal D,\Ab^\Xb}(\mathfrak l_k+\epsilon)$, including counting multiplicity. 
\end{proof}

\begin{remark}
It would be interesting to understand, like in the bipartite situation, the link between the zero deficiency property
\[
  1 - \dim \ker (T_{\Ab^\Xb} -\mathfrak l_k)+ \Mor\bigl(\DN_{\mathcal D,\Ab^\Xb}(\mathfrak l_k+\epsilon)\bigr),
\]
and the minimal partition property.
   
It is mentioned in~\cite[Remark~5.2]{HH:2015b} that if we have a minimal $k$-partition then we are in the Courant sharp situation  for the corresponding AB Hamiltonian $T_{\Ab^\Xb}$, i.e. it has the zero deficiency property.

The converse is true as recalled above for a bipartite partition but wrong in general. A counterexample is given for the square and $k=5$ in~\cite[Fig.~19]{BH}, which is kindly reproduced in Figure~\ref{fig:5pcc}. We have on the left a $5$-partition with one critical odd point which is the nodal partition of the $5$-th eigenfunction  of its associated AB operator, but  is not minimal. We have on the right a $5$-partition with four  critical odd points  which is the nodal partition of the  $5$-th eigenfunction of its associated AB operator, which is not minimal. It is conjectured that  a  minimal 5-partition is indeed obtained for the middle  configuration  with also  four odd  critical points.

\begin{figure}[htbp]
\centering
\includegraphics[height=3.5cm]{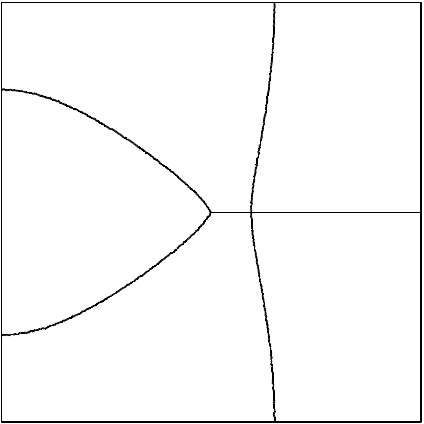} 
\hskip0.5cm
\includegraphics[height=3.5cm]{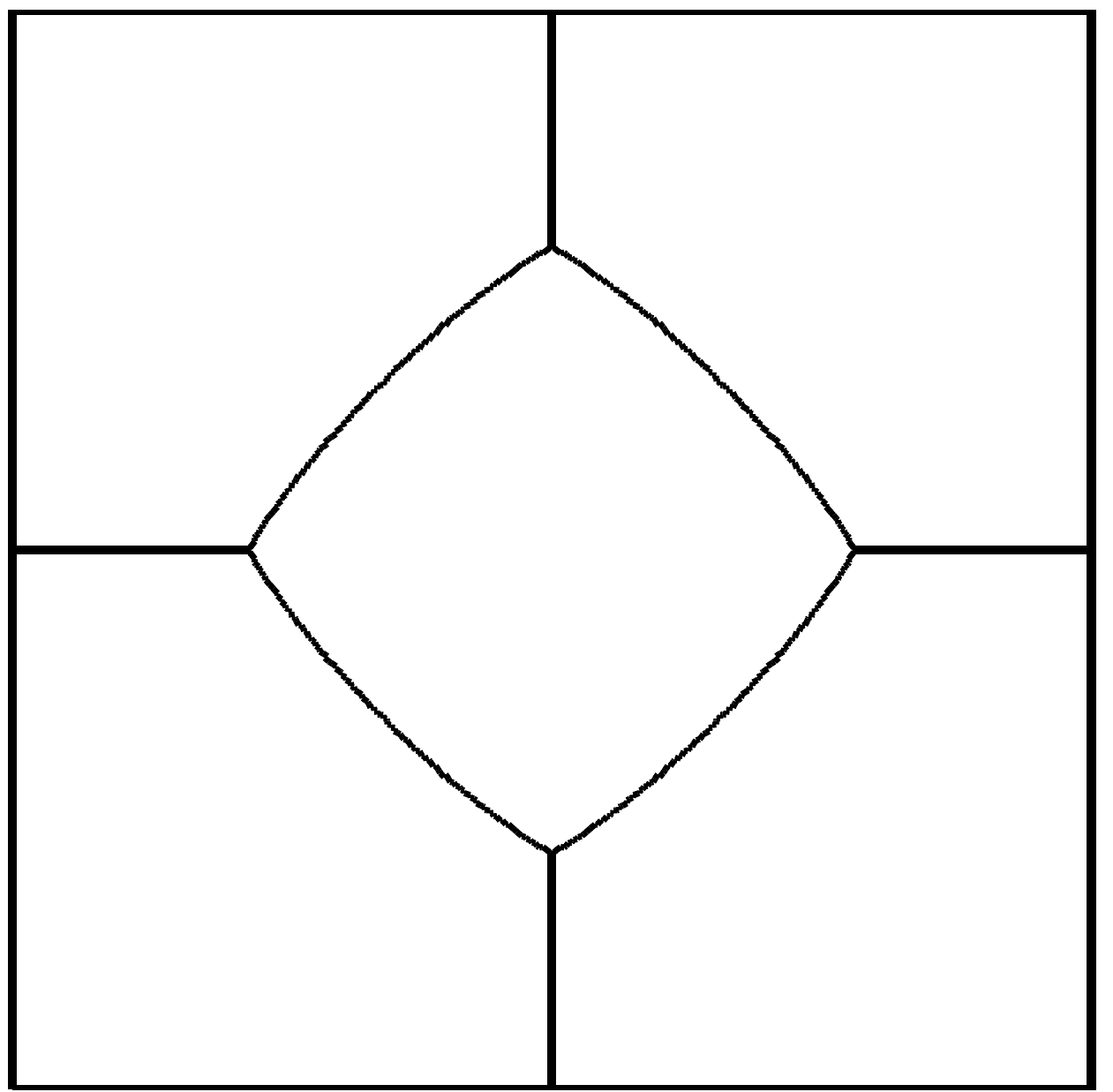} 
\hskip0.5cm
\includegraphics[height=3.5cm]{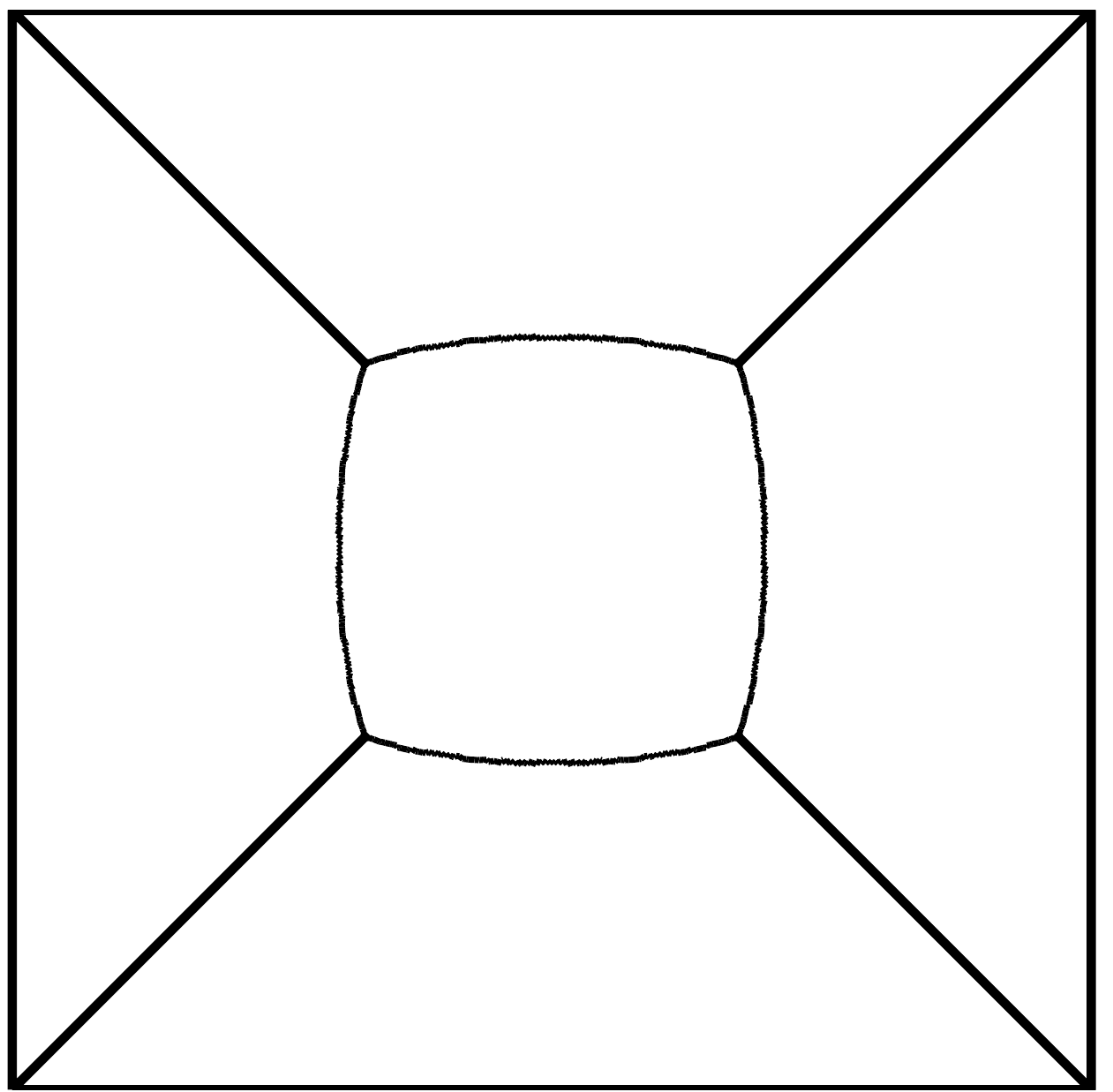} 
\caption{Three $5$-equipartitions satisfying the PCC, with $0$ deficiency index. 
The middle one has minimal energy among these three.}
\label{fig:5pcc}
\end{figure}
\end{remark}

\subsection{The Berkolaiko--Cox--Marzuola construction through double covering lifting}
In many of the papers analyzing minimal partitions, the authors refer to a double covering argument. This point of view (which appeared first in~\cite{HHOO} in the case of domains with holes) is essentially equivalent to the Aharonov approach. We just mention the main lines of the argument. One can, in an abstract way, construct a double covering manifold $\widetilde \Omega:=\dot \Omega^{\bf X}_{\mathcal R}$ above $\dot \Omega^{\bf X}$. One can then lift the initial spectral problem to one for the Laplace operator on this new (singular) (see \cite{AFT} and \cite{HHOT}) $\widetilde \Omega $. In this lifting, the $K^{{\Xb}}$-real eigenfunctions become eigenfunctions which are real and antisymmetric with respect to the deck map (exchanging two points having the same  projection on $\dot \Omega^{\bf X}$).

In the case of the disk,  with one AB half normalized flux in the center, the construction is equivalent to considering the angular variable $\theta \in (0,4\pi)$, and the deck map corresponds to the translation by $2\pi$. The nodal set of the $6$-th eigenfunction gives by projection the Mercedes star and the $11$-th eigenvalue (which is the $5$-th in the space of antiperiodic functions) gives by projection the candidate for a minimal three-partition.

Starting from an eigenfunction $u$ of $T_{\Ab^\Xb}$ with zero set $\Gamma$, the idea is now to apply the construction of Berkolaiko--Cox--Marzuola to the Laplacian on $\widetilde \Omega $ and to the lifted eigenfunction $\tilde u$, having in mind that this is an antisymmetric eigenfunction. The zero-set of $\tilde u$ is $\widetilde\Gamma=\Pi^{-1} (\Gamma)$.  One should then interpret the quantities for the covering $\widetilde \Omega$ in term of  quantities on $\Omega$.

Hence we should define the Dirichlet to Neumann attached to the Laplacian on  $\widetilde \Omega $ and $\widetilde \Gamma$ and reinterpret it when restricted to antisymmetric functions on $\widetilde \Gamma$. The spectrum of $T_{\Ab^\Xb}$ consists of the eigenvalues corresponding to the antisymmetric eigenfunctions of $-\Delta$. Hence the labelling of the eigenvalue of $T_{\Ab^\Xb}$ corresponds to the labelling of $-\Delta $ restricted to the antisymmetric space.

\section{The cutting construction for general regular PCC-equipartitions}\label{sec:cutting}
We describe below the approach to make cuts in the domain $\Omega$ along some of the subpaths in $\Gamma$. In this way, we avoid to discuss the artificial singularities introduced with Aharonov--Bohm operators. We start by explaining how this procedure works in a simple case.

\subsection{Example: the Mercedes star.}\label{sMS}
We first consider the case when we have in $\Omega$ a $3$-partition, with only one critical point (which has the topology of the Mercedes star). We can assume that $\Omega$ is the disk, and that we have the critical point at $0\in \Omega$. We further denote by  $D_1$, $D_2$ and $D_3$ the components of the $3$-partition, and by $\Gamma_1$, $\Gamma_2$ and $\Gamma_3$ the three branches of the star, as in Figure~\ref{fig:Mercedes}. Starting from the \enquote{formal} Aharonov--Bohm point of view, we try to eliminate the reference to this operator by using a suitable square root $\exp(\i \phi_{X}/2)$.

\begin{figure}[htb]
\centering
\subcaptionbox{\label{Mercedes-a}}{\includegraphics[page=3]{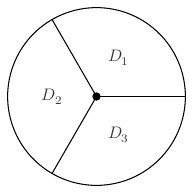}}
\hskip 1cm
\subcaptionbox{\label{Mercedes-b}}{\includegraphics[page=4]{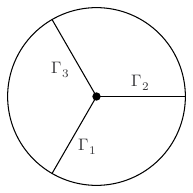}}
\hskip 1cm
\subcaptionbox{\label{Mercedes-c}}{\includegraphics[page=25]{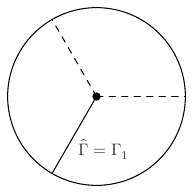}}
\caption{The unit disk with \subref{Mercedes-a} a simple $3$-equipartition 
\subref{Mercedes-b} the boundary of this $3$-equipartition. In \subref{Mercedes-c} the slitted domain is shown.}
\label{fig:Mercedes}
\end{figure}

The idea is to look at our Robin form as a sesquilinear form on a functional  space defined on $\Omega \setminus \Gamma_1$. We note that  $\exp(\i \phi_{X}/2)$ can be well defined as a univalued function on $\Omega \setminus \Gamma_1$. If $u\in H_{0,\Ab}^1(\Omega)$, $\hat u:= \exp(-\i \phi_{X}/2)u$ belongs to $H^1(\Omega\setminus \Gamma_1)$ with boundary condition $\hat u|_{\partial \Omega}=0$ and $\hat u|_{\Gamma_1^-} = -\hat u|_{\Gamma_1^+}$. Here the $-$ and $+$ indicate that we have introduced an orientation on $\Gamma_1$, and that we we approach the curve $\Gamma_1$ from different sides, say left for $-$ and $+$ for right). We denote this space as $\widehat H_{0}^1(\Omega \setminus \Gamma_1)$.

On this new Sobolev space $\widehat H_{0}^1(\Omega \setminus \Gamma_1)$, we get the sesquilinear form
\[
\begin{aligned}
 (\hat u,\hat v)\mapsto \hat {\mathfrak t}_\sigma (\hat u, \hat v) 
 & = \sum_{j=1}^3 \int_{D_j} \nabla  \hat u  \cdot \overline{\nabla \hat  v} \,dx 
 + \sigma \int_\Gamma \hat u\overline{\hat v}\,ds_\Gamma\\ 
 & = \int_{\Omega \setminus \Gamma_1} \nabla  \hat u  \cdot \overline{\nabla \hat  v}\, dx\,  
 + \sigma \int_{\Gamma} \hat u\, \overline{\hat v}\,ds_\Gamma\,.
 \end{aligned}
\]
Here we note that on $\Gamma_1$ the left trace of $ \hat u\, \overline{\hat v}$ equals the right trace of $ \hat u\, \overline{\hat v}$. We need to describe the transmission obtained on $\Gamma$, the operator being the standard Laplacian. For this purpose, we write $\hat u|_{D_j}=\hat u_j$ ($j=1,2,3$) and redescribe $\widehat H_{0}^1(\Omega \setminus \Gamma_1)$, and we need to give the new transmission relations through $\Gamma_1$, $\Gamma_2$ and $\Gamma_3$.

Introducing (in an arbitrary way) orientations on the curves, we can talk about left traces and right traces, and we denote the corresponding outward pointing normal derivatives by $\partial_{\nu^k_{\pm}}$. The transmission conditions are unchanged on $\Gamma_2$ and $\Gamma_3$, where we did not make the cut. This means that
\[
\partial_{\nu_-^2} \hat u_1  + \partial_{\nu_+^2}  \hat u_3 
=  -\sigma  \hat u_{3},\quad \hat u_3=\hat u_1,\quad \text{on $\Gamma_2$,}
\] 
and
\[
\partial_{\nu^1_-} \hat u_2  + \partial_{\nu^1_+}  \hat u_1 
=  -\sigma  \hat u_{2},\quad \hat u_2= \hat u_1,\quad \text{on $\Gamma_3$.}
\]
On the path where we did the cut, however, we obtain a Robin-like boundary condition
\[
-\partial_{\nu_-^3} \hat u_2  + \partial_{\nu_+^3}  \hat u_3 
=  -\sigma  \hat u_{3},\quad \hat u_3= - \hat u_2,\quad \text{on $\Gamma_1$.}
\]
Thus, associated with the quadratic form $\hat{\mathfrak t}_\sigma$ we associate a self-adjoint Laplace operator $\widehat T_{\Gamma, \widehat \Gamma,\sigma}$, where we for future reference has defined $\widehat \Gamma$ as the part of $\Gamma$ that is removed from $\Omega$, in our case
\[
\widehat \Gamma= \Gamma_1.
\]

For $\sigma =0$, we get the Laplace operator $\laplace_{\Omega \setminus \widehat\Gamma}$ playing the role of the Dirichlet realization in~\cite{BeCoMa} and the AB Hamiltonian $T_{\Ab^\Xb}$ in the magnetic Laplacian case that we discussed in the previous section. For $\sigma =+\infty$, we recover the Dirichlet Laplacian on the disjoint union of the $D_i$'s.

We next define the Dirichlet-to-Neumann type operator that replaces $\Lambda_-(\epsilon) +\Lambda_{+}(\epsilon)$ in~\cite{BeCoMa}. It is again an operator from $H^{1/2}(\Gamma)$ to $H^{-1/2} (\Gamma)$, and we will this time denote it by $\DN_{\Gamma,\widehat{\Gamma}}(\lambda)$.

Let $h$ belong to $H^{1/2}(\Gamma)$. We modify the $R_{\Gamma\shortto\partial D_i}$ operators to take into account our new conditions at $\widehat\Gamma=\Gamma_1$. To do this, we first recall that we have introduced orientations on each $\Gamma_i$. With our choice above, this means that we have chosen to flip the sign on $\Gamma_1$ when we are on the $D_2$ side of it. Thus, we set for $i=1$ and $i=3$
\[
  \widehat{R}_{\Gamma\shortto \partial D_i}
  = R_{\Gamma\shortto \partial D_i},
\]
while we for $i=2$ instead set
\[
  \widehat{R}_{\Gamma\shortto \partial D_2}h
  =
  \begin{cases}
    h & \text{on $\Gamma_3$},\\
    -h & \text{on $\Gamma_1$},\\
    0 & \text{on $\partial\Omega\cap\partial D_2$}.
  \end{cases}  
\]
Then, with the classical Dirichlet-to-Neumann operator in $D_i$ denoted by $\DN_{\partial D_i}$, we set
\begin{equation}\label{eq:DNslitting}
  \DN_{\Gamma,\widehat\Gamma}(\lambda)h=\sum_{i=1}^3 \widehat{R}_{\partial D_i\shortto\Gamma} \DN_{\partial D_i} \widehat{R}_{\Gamma\shortto \partial D_i}h.
\end{equation}
The operator $\DN_{\Gamma,\widehat\Gamma}(\lambda)$ is self-adjoint, the proof being very similar to the proof of Proposition~\ref{prop:DNsa}.

In this formalism, we denote by $\Mor(\DN_{\Gamma,\widehat\Gamma}(\lambda))$ the number of negative eigenvalues of $\DN_{\Gamma,\widehat\Gamma}(\lambda)$. We introduce the defect $\widehat{\Def}(\mathcal D,\widehat\Gamma)$ of the partition $\mathcal D$ as
\[
\widehat{\Def}(\mathcal D):=\hat \ell(\mathcal D,\widehat\Gamma)-k(\mathcal D),
\]
where $\hat\ell(\mathcal D)$ denotes the minimal labelling of the eigenvalue $\hat{\mathfrak l}_k$ of the Hamiltonian $T_{\Gamma,\widehat\Gamma,0}=\laplace_{\Omega\setminus\widehat\Gamma}$. We can reformulate Theorem~\ref{thm:main} in the following way:

\begin{theorem}\label{thm:flowa}
Let $\mathcal D$ a regular $k$-equipartition of  a simply connected domain $\Omega$ satisfying the PCC with energy $\mathfrak l_k=\mathfrak l_k(\Omega)$. Then, for sufficiently small $\epsilon >0$, 
\begin{equation}\label{eq:5.1}
\widehat {\Def} (\mathcal D)
= 1 - \dim \ker ( \laplace_{\Omega\setminus \widehat \Gamma} -(\mathfrak l_k+\epsilon))
+  \Mor(\DN_{\Gamma,\widehat\Gamma}(\mathfrak l_k+\epsilon)).
\end{equation}
\end{theorem}

The proof of this theorem in the case of the Mercedes star, where we have done our construction of the operators, follows the lines of the proof of Theorem~\ref{thm:main}. Note that due to the symmetries one can have in this case a nicer more explicit expression for $\DN_{\Gamma,\widehat\Gamma}(\mathfrak l_k+\epsilon)$ (see the analysis on the circle from Appendix~\ref{ex:circle}). The question is now how to choose $\widehat\Gamma$ in the general situation and then also how to define $\DN_{\Gamma,\widehat\Gamma}(\lambda)$.

\subsection{The general case}
Our aim is now to generalize the construction we have done for the Mercedes star.

\subsubsection{The choice of $\widehat \Gamma$}
This is indeed quite analogous to what is done when we want to define the square  root of $z\mapsto (z-z_1)(z-z_2)\ldots (z-z_\ell)$ in a maximal domain of $\mathbb C$. By defining branch cuts, one can recover the double covering by gluing the two sheets along these branch cuts. 
 
In our case, we have in a addition a boundary set $\Gamma$ containing the \enquote{odd} points $X_1$, $\ldots$, $X_\ell$ in $\Omega$ and what we have to prove is that $\Gamma$ contains a closed subset $\widehat \Gamma$ corresponding to the branch cuts, which is minimal, in a sense described below. These branch cuts are either connecting inside $\Gamma$ one odd point to one point of the boundary or two odd points.

We should have the property that we can then construct a square root of $\exp (\i\phi)$ denoted by $\exp (\i\phi/2)$ which is univalued on $\Omega \setminus \widehat \Gamma$ and maximal in the sense that it can not be extended to a larger open set. The set $\widehat \Gamma$ will in general not be unique, but all we need is the existence. A natural notion was introduced in~\cite{HHOO} called the slitting property and the only change is that holes are replaced here by points (or) poles.

For a given $\Xb= (X_1,\ldots,X_\ell)$ in $\Omega^\ell$ with distinct $X_i$, we say that a closed set $N$ \emph{slits} $\overline\Omega$ with singularities at $\Xb$ if: 
\begin{itemize}[--]
\item $N$ is a weakly regular closed set in the sense of 
Section~\ref{sec:regularityassumptions};
\item $X^{\mathrm{odd}}(N)=(X_1,\ldots, X_\ell)$;
\item $\Omega \setminus  N$ is connected.
\end{itemize}

This definition was introduced to characterize the properties of the nodal domain of the ground state of $T_{\Ab}$. 

\begin{figure}[htb]
\centering
\subcaptionbox{\label{slitting-a}}{\includegraphics[page=7]{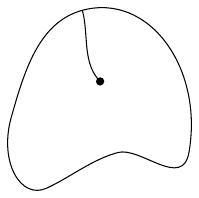}}
\hskip 0.75cm
\subcaptionbox{\label{slitting-b}}{\includegraphics[page=8]{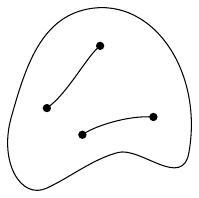}}
\hskip 0.75cm
\subcaptionbox{\label{slitting-c}}{\includegraphics[page=9]{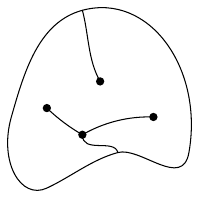}}
\\[1cm]
\subcaptionbox{\label{slitting-d}}{\includegraphics[page=10]{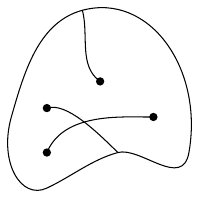}}
\hskip 0.75cm
\subcaptionbox{\label{slitting-e}}{\includegraphics[page=11]{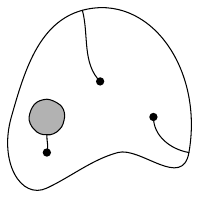}}
\hskip 0.75cm
\subcaptionbox{\label{slitting-f}}{\includegraphics[page=12]{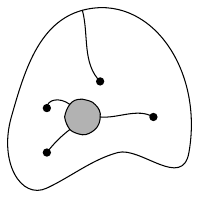}}
\caption{(Topological) slitting examples,
\subref{slitting-a} 
\subref{slitting-b}, 
\subref{slitting-c},
\subref{slitting-d},
\subref{slitting-e},
\subref{slitting-f}.}
\label{fig:slitting}
\end{figure}

Figure~\ref{fig:slitting} is inspired by~\cite[Fig.~1]{HHOO}, and shows some examples of regions which are slitting (but replace the holes in the picture by points in our case). Note that crossing points at even points are permitted (see  Figure~\ref{slitting-d}). Note also that for the $\ell=1$ case (Figure~\ref{slitting-a}), a set which slits $\overline\Omega$ consists of one line which joins the outer boundary of $\Omega$ to the pole (Mercedes situation). We have explained above the no-hole situation. The case with holes is treated in the same way, once we have selected some \enquote{odd} holes, and placed one pole $X'_j$ in each of these odd holes. If a collection of paths slits a region then no sub- or supercollection of these paths can also slit the region.

In this formalism the main result is

\begin{proposition}\label{prop:slitting}
Assume that $\Gamma$ is regular with odd points $X^{\mathrm{odd}}(\Gamma)=\Xb$ in the interior of $\Omega$ and odd points  $\partial \Omega^{\mathrm{odd}}(\Gamma)$ on the boundary of $\Omega$. Then it contains a slitting set $\widehat \Gamma$ that preserves the odd points, i.e. such that $X^{\mathrm{odd}}(\widehat\Gamma)=X^{\mathrm{odd}}(\Gamma)=\Xb$ and  $\partial \Omega^{\mathrm{odd}}(\widehat \Gamma)= \partial \Omega^{\mathrm{odd}}(\Gamma)$.
\end{proposition}

Then, once we have this slitting property, we have 

\begin{proposition}
Under the assumptions on $\Gamma$ from Proposition~\ref{prop:slitting}, there exists $\widehat \Gamma$ such that $X^{\mathrm{odd}}(\Gamma) \subset \widehat \Gamma \subset \Gamma$ and such that there exists in $\Omega\setminus \widehat \Gamma$ a univalued regular square root of $\exp(\i\Phi_\Xb)$ which is maximal in the sense that it cannot be extended as a univalued regular function in an open set in $\Omega$ containing strictly $\Omega\setminus \widehat \Gamma$.
\end{proposition}

It turns out that the statement in Proposition~\ref{prop:slitting} is has a natural formulation in terms of graph theory. This corresponds indeed in the nodal case to the notion of nodal graph (see for example~\cite[Subsection 3.1]{KKP}). We have a graph contained in $\overline\Omega \subset \mathbb R^2$. The boundary points and the singular points of $\Gamma$ are the vertices and the regular arcs of $\Gamma$ are the edges. The vertices of $\partial \Omega$  and the odd vertices (an odd number of edges arrive at the vertex) in $\Omega$ play a special role. We can also define the notion of odd hole by determining the parity of the number of edges arriving at the boundary of the hole. Moreover, $\widehat \Gamma$ can be considered as a subgraph of $\Gamma$. The graph translation of Proposition~\ref{prop:slitting} is:

\begin{lemma}
If $\Gamma$ is a graph in $\overline\Omega$ with given \enquote{odd} set of  vertices $X^{\mathrm{odd}}(\Gamma)= (X_1,\ldots, X_\ell)$ and given \enquote{odd}  holes, then there exists a subgraph $\widehat\Gamma$ with the same \enquote{odd}  sets such that $\Omega \setminus \widehat \Gamma$ is connected.
\end{lemma}

\begin{proof}[Proof (given by G.~Berkolaiko)]
We first consider the case with a simply connected domain $\Omega$. It is better to identify all the points of $\partial\Omega$ and to look a the new graph as a graph $\widetilde \Gamma$ on the sphere, $\partial \Omega$ being the north pole $P$ and $\Omega$ being the $\mathbb S_2\setminus \{P\}$. To get the connexity property, it is enough to destroy all the cycles on the graph. It is now enough to observe that if there is a cycle we can delete all the elements of the cycle. But at each vertex of the cycle, only two edges belonging to the cycle arrive. Hence when destroying a cycle, this always preserves the odd vertices and the even vertices. Finally, we observe that no odd vertex can disappear when deleting a cycle. This case is exemplified in Figure~\ref{fig:graphslitting1}.

\begin{figure}[htb]
\centering
\makebox[\textwidth][c]{%
\subcaptionbox{\label{graphslitting1-a}}{\includegraphics[page=18]{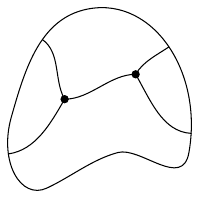}}
\hskip 1cm
\subcaptionbox{\label{graphslitting1-b}}{\includegraphics[page=19]{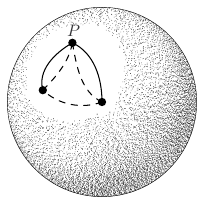}}
\hskip 1cm
\subcaptionbox{\label{graphslitting1-c}}{\includegraphics[page=20]{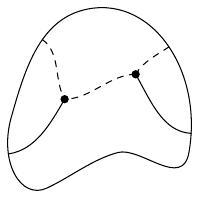}}
}
\caption{\subref{graphslitting1-a} A domain $\Omega$ with a partition.
\subref{graphslitting1-b} The constructed graph on the sphere. The boundary
of $\Omega$ is mapped to the point $P$. We remove one loop (dashed).
\subref{graphslitting1-c} Back in $\Omega$, the removed loop is dashed, and $\widehat\Gamma$ remains.}
\label{fig:graphslitting1}
\end{figure}

In the case with holes, we identify each component of $\partial \Omega$ with a  point and look at a new graph $\widetilde \Gamma$ on the sphere with $\Omega$ being $\mathbb S_2\setminus \{P_1,\dots, P_{m}\}$. Each $P_i$ corresponds to a component of $\partial \Omega$. Then it suffices to think in the previous proof that the points $P_\ell$ corresponding to odd boundary components are odd vertices. This situation is exemplified in Figure~\ref{fig:graphslitting2}.
\end{proof}

\begin{figure}[htb]
\centering
\makebox[\textwidth][c]{%
\subcaptionbox{\label{graphslitting2-a}}{\includegraphics[page=21]{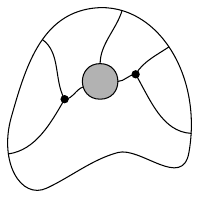}}
\hskip 1cm
\subcaptionbox{\label{graphslitting2-b}}{\includegraphics[page=22]{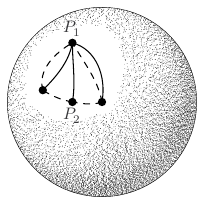}}
\hskip 1cm
\subcaptionbox{\label{graphslitting2-c}}{\includegraphics[page=23]{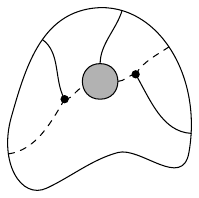}}
}
\caption{\subref{graphslitting2-a} A non-simply connected domain $\Omega$ with a partition.
\subref{graphslitting2-b} The corresponding graph on the sphere. 
The outer boundary is identified
at $P_1$ while the inner boundary is identified at $P_2$. We remove one loop (dashed).
\subref{graphslitting2-c} Back in $\Omega$, this means that we removed the
dashed curves.}
\label{fig:graphslitting2}
\end{figure}

\subsubsection{Proof of Theorem~\ref{thm:flowa}} With the slitting lemma at hand, we have a general method to define $\widehat\Gamma$, and then we can complete the proof in the general setting, by following what was done in the case of the Mercedes star. Thus, we introduce a Sobolev space associated with the pair $(\widehat \Gamma, \Gamma \setminus \widehat \Gamma)$. 

We note that $\widehat \Gamma$ is a union of regular curves $\hat \gamma_\ell$,  ending at critical points, and we can choose an orientation of $\hat\gamma_\ell$ so that, locally, in the neighborhood $D(x,r)$ of an interior point $x\in \hat\gamma_\ell$, we can write $D(x,r) \setminus \hat\gamma_\ell= D^+(x,r)\cup D^-(x,r)$ permitting to define a trace on the left and on the right.

Starting from $H^1(\Omega \setminus \widehat \Gamma)$, we introduce
\[
\widehat H_0^1(\Omega \setminus \widehat \Gamma):=\{ u\in H^1( \Omega \setminus \widehat \Gamma)\,,\, u|_{\hat \gamma_\ell^+} = - u|_{\hat \gamma_\ell^-}\,,\, u|_{\partial \Omega} =0\}\,.
\]
On this  Sobolev space $\widehat H_{0}^1(\Omega \setminus \hat \Gamma)$, we define the sesquilinear form
\[
\begin{aligned}
 (\hat u,\hat v)\mapsto \hat {\mathfrak t}_\sigma (\hat u, \hat v) 
 & = \sum_{i=1}^3 \int_{D_i} \nabla  \hat u  \cdot \overline{\nabla \hat  v} \,dx
 + \sigma \int_\Gamma \hat u\overline{\hat v}\,ds_\Gamma\\ 
 & = \int_{\Omega \setminus \widehat \Gamma} \nabla  \hat u  \cdot \overline{\nabla \hat  v}\, dx
  + \sigma \int_{\Gamma} \hat u\, \overline{\hat v}\,ds_\Gamma\,.
 \end{aligned}
\]
We can then associate, via the Lax--Milgram theorem, to this sesquilinear form a realization $\widehat T_{\Gamma,\widehat \Gamma,\sigma}$ of the Laplacian in $\Omega \setminus \widehat \Gamma$ with $\sigma$ transmission properties on $\Gamma \setminus \widehat \Gamma$, Dirichlet condition on $\partial \Omega$ and $\sigma$-Robin like condition on $\widehat\Gamma$. As in the Mercedes-star case, $\laplace_{\Omega \setminus \widehat \Gamma}$ corresponds to $\sigma=0$.

It remains to detail our definition of $\DN_{\Gamma,\widehat\Gamma}(\lambda)$. Following what we did in the Mercedes-star case, the only point is to have a clear definition of the operators $\widehat{R}_{\Gamma\shortto \partial D_i}$. We change signs where appropriate, and set
\[
  \widehat{R}_{\Gamma\shortto \partial D_i}h
  =
  \begin{cases}
    h & \text{on $\Gamma\setminus\widehat{\Gamma}_i^{-}$},\\
    -h & \text{on $\widehat{\Gamma}_i^-$},\\
    0 & \text{on $\partial\Omega\cap\partial D_i$}.
  \end{cases}  
\]
Here $\widehat{\Gamma}_i^{-}$ denotes the parts of $\widehat{\Gamma}$ where we have chosen the left trace as seen from $D_i$. The definition of $\DN_{\Gamma,\widehat\Gamma}(\lambda)$ is then given as in~\eqref{eq:DNslitting}, but with a sum from $1$ to $k$. The proof of Theorem~\ref{thm:flowa} is then completed in the same way as in~\cite{BeCoMa}.

\subsubsection{Comparison between two constructions}\label{sec:comparison}
If $\mathcal D$ is an equipartition with boundary set $\Gamma$, we can observe that $\Omega \setminus{\widehat\Gamma}$ is a bipartite equipartition. Moreover, if it satisfies the PCC, it is a nodal partition. Finally, if $\mathcal D$ is a minimal partition then it is a Courant sharp nodal partition in $\Omega \setminus \widehat \Gamma$ (see~\cite{BHV} where this argument is used for the analysis of the Hexagonal conjecture, see also Figure~\ref{fig:hexagon}). 
\begin{figure}[htb]
  \centering
  \subcaptionbox{\label{hexagon-a}}{\includegraphics[page=24]{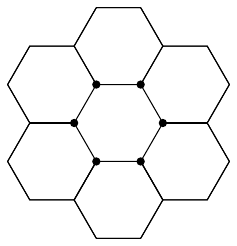}}
  \hskip 1cm
  \subcaptionbox{\label{hexagon-b}}{\includegraphics[page=17]{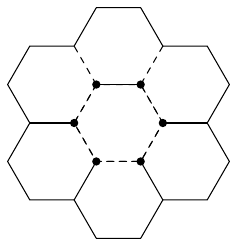}}
  \caption{The slitting example $H_{12}^7$ of \cite[Figure~24]{BHV}.
  This is an example of a (conjectured) effective minimal $7$-partition of  $\Omega$, where $\Omega$ is the interior of the union of the closed seven hexagons, the minimal partition consisting of the seven hexagons. One can verify that it has the PCC property.}
  \label{fig:hexagon}
\end{figure}
It is then natural to compare our construction relative to $\mathcal D$ (seen as  a partition of $\Omega$) with the Berkolaiko--Cox--Marzuola construction associated with $\mathcal D$ seen as a nodal partition in $\Omega \setminus \widehat \Gamma$.  The difference is that in the second case, we restrict the first construction to  elements which vanish on $\widehat \Gamma$ and then project on $H^{-1/2}(\Gamma \setminus \widehat \Gamma)$. Coming back to the definitions,  it is then immediate to see that an eigenvalue of the Dirichlet to Neumann second operator is actually an eigenvalue of the first one. It is then interesting to compare the two formulas~\eqref{eq:5.1} and~\eqref{eq:BeCoMa} for the pair $(\Omega\setminus \widehat \Gamma,\Gamma\setminus \widehat \Gamma)$.

\section*{Acknowledgements}
The first author would like to thank G.~Berkolaiko for comments and his precious help in graph theory. We also thank T.~Hoffmann-Ostenhof and G.~Cox for useful comments.

\appendix
\section{Examples: Equipartitions of the unit circle}\label{ex:circle}
Even though we will consider domains in $\mathbb R^2$, we start by doing some calculations for the unit circle. We assume that $N$ is odd (even $N$ correspond to the nodal case) and consider $N$-equipartitions $\mathcal D$ (see Figure~\ref{fig:unitcircle} for the cases $N=3$ and $N=5$),
\[
k(\mathcal D)=N
\]
of the unit circle and the angular part of the Laplacian, $-\frac{d^2}{d\theta^2}$, with Dirichlet conditions at each sub-dividing point. Each interval have length $\Theta=2\pi/N$, and the smallest eigenvalue---the energy of the partition---is given by 
\[
  \Lambda(\mathcal D)=(N/2)^2.
\]

\begin{figure}[htb]
\centering
\subcaptionbox{\label{unitcircle-a}}{\includegraphics[page=1]{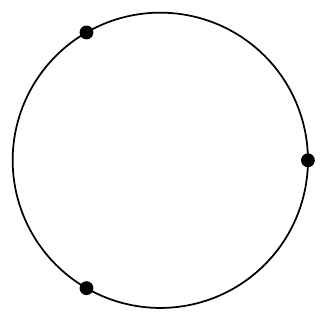}}
\hskip 1cm
\subcaptionbox{\label{unitcircle-b}}{\includegraphics[page=2]{figures.pdf}}
\caption{The unit circle with \subref{unitcircle-a} the $3$-partition and \subref{unitcircle-b} the $5$-partition.}
\label{fig:unitcircle}
\end{figure}

The corresponding magnetic operator on the circle is given by
\[
T=-\Bigl(\frac d {d\theta} -\frac{\i\pi}{2}\Bigr)^2\,,
\]
and its spectrum consist of eigenvalues $\bigl\{\bigl(\frac{2n-1}{2}\bigr)^2\bigr\}_{n=1}^{+\infty}$, each with multiplicity two,
\[
\dim\ker\Bigl[T-\Bigl(\frac{2n-1}{2}\Bigr)^2\Bigr]=2.
\]
In particular, the minimal label $\ell(\mathcal D)$ of the eigenvalue $\Lambda(\mathcal D)=(N/2)^2$ is given by
\[
\ell(\mathcal D)=N.
\]
We are going to test the formula
\begin{equation}
\label{eq:proposedformula}
\ell(\mathcal D)-k(\mathcal D)=1-\dim\ker\bigr(T-\Lambda(\mathcal D)\bigl)+\Mor(\mathcal M_\lambda),
\end{equation}
where $\Mor(\mathcal M_\lambda)$ denotes the number of negative eigenvalues of a Dirichlet-to-Neumann operator $\mathcal M_\lambda$, discussed below. In fact, since we just saw that $\ell(\mathcal D)=N$, $k(\mathcal D)=N$, $\dim\ker\bigr(T-\Lambda(\mathcal D)\bigl)=2$, we need to check that
\[
  \Mor(\mathcal M_\lambda)=1.
\]
This is similar to the setting for Quantum graphs. In~\cite{FrGr} the number of negative eigenvalues of a Dirichlet-to-Neumann operator of a graph Laplacian corresponding to energy $\lambda$ is calculated as a difference between the number of eigenvalues of the corresponding Neumann and Dirichlet graph laplacians less than $\lambda$. But graphs with loops are excluded.

First we compute the Dirichlet-to-Neumann operator and the associated $2\times 2$ matrix $M_\lambda$ which associates with the solution $u$ of 
\[
-\frac{d^2}{d\theta^2} u =\lambda u\,,\, u(0)=u_0 \,,\,u(\Theta)=u_1\,,
\]
the pair 
\[(v_0,v_1)=  (-u'(0), u'(\Theta))\,.
\]
This leads to
\[
\begin{bmatrix}
v_0\\
v_1
\end{bmatrix}
=
M_\lambda
\begin{bmatrix}
u_0\\
u_1
\end{bmatrix},
\]
where $M_\lambda$ is the matrix
\[
 M_\lambda = 
\begin{bmatrix} 
\sqrt{\lambda}\cot(\sqrt{\lambda} \Theta)  & -\dfrac{\sqrt{\lambda}}{\sin (\sqrt{\lambda}\Theta)}\\ 
-\dfrac{\sqrt{\lambda}}{\sin (\sqrt{\lambda}\Theta)}  &  \sqrt{\lambda} \cot (\sqrt{\lambda}\Theta)
\end{bmatrix}
= 
 \begin{bmatrix}
\alpha (\lambda)  & \beta(\lambda) \\
\beta(\lambda) &  \alpha(\lambda)
\end{bmatrix},
\]
and $\alpha(\lambda)$ and $\beta(\lambda)$ are defined via the equation above. We continue in the same way along the circle. With $(u_k,u_{k+1})=(u(k\Theta),u((k+1)\Theta))$ and $(v_k,v_{k+1})=(-u'(k\Theta),u'((k+1)\Theta))$, we find that
\[
\begin{bmatrix}
v_k\\
v_{k+1}
\end{bmatrix}
=
M_\lambda
\begin{bmatrix}
u_k\\
u_{k+1}
\end{bmatrix},\qquad 0\leq k\leq N-1.
\]
But when we come to $(u_N,v_N)$ we have walked around the circle, and are back at the point we started. We replace $(u_N,v_N)$ by $(-u_0,-v_0)$. Thus, we find that the $N\times N$ matrix $\mathcal M_\lambda$, that associates with $(u_0,u_1,\ldots,u_{N-1})$ the $N$-tuple $(v_0,v_1,\ldots,v_{N-1})$, is given by
\[
 \mathcal M_\lambda := 
\frac12
\begin{bmatrix}
2 \alpha (\lambda) &\beta(\lambda)     &              0  & 0                & \cdots         &-\beta(\lambda) \\
\beta(\lambda)     &2\alpha(\lambda)   & \beta(\lambda)  & 0                & \cdots         & 0 \\
0                  & \beta(\lambda)    &2\alpha(\lambda) & \beta(\lambda)   & \cdots         & 0 \\
0                  & 0                 & \beta(\lambda)  & 2\alpha(\lambda) & \cdots         & \vdots \\
\vdots             & \vdots            & \vdots          & \ddots           & \ddots         & \beta(\lambda) \\
- \beta(\lambda)   & 0                 & 0               & \cdots           & \beta(\lambda) & 2 \alpha(\lambda) 
\end{bmatrix}\,.
\]
Thus, $\mathcal M_\lambda$ has $\alpha$ on the main diagonal, $\beta/2$ on the two subdiagonals, and $-\beta/2$ in the corners $(1,N)$ and $(N,1)$. The eigenvalues of $\mathcal M_\lambda$ are given by
\begin{equation} \label{defmuk}
\mu_k = \alpha - \beta \cos (2k \pi/ N)\,,\,k=0,\dots, N-1\,.
\end{equation}
Hence the lowest one is $\mu_0=\alpha(\lambda) - \beta(\lambda)$, and this eigenvalue is negative if $\sqrt{\lambda}=N/2+\epsilon$, with $\epsilon>0$ sufficiently small.
 
To analyze the positivity of the other eigenvalues, it suffices to study the sign of $\mu_1$. After division by $\beta(\lambda)$, we need to consider the sign of
\[
\delta_1:=  - \cos (2\pi \sqrt{\lambda}  /N)- \cos (2\pi / N).
\]
If we take $\sqrt{\lambda} = N/2 +\epsilon$, we have $\delta_1(\epsilon) = \cos ( 2 \pi \epsilon/N) - \cos (2\pi / N) >0$ for $\epsilon >0$ small enough.

We conclude that if $\sqrt{\lambda}=N/2+\epsilon$, with $\epsilon>0$ sufficiently small, then the matrix $\mathcal M_\lambda$ has exactly $1$ negative eigenvalue. This means that Formula~\eqref{eq:proposedformula} is indeed  true.

Finally, we continue the discussion in Section~\ref{sec:comparison}. Now $\widehat \Gamma$ is just one point, say $\theta=0$. The construction in~\cite{BeCoMa} leads to an $(N-1)\times (N-1)$ matrix obtained by taking $u_0=0$ and forgetting $v_0$. This leads to the matrix
\[
 \mathcal M_\lambda^0 := 
\frac12
\begin{bmatrix}
2\alpha(\lambda)   & \beta(\lambda)  & 0                & \cdots         & 0 \\
 \beta(\lambda)    &2\alpha(\lambda) & \beta(\lambda)   & \cdots         & 0 \\
0                 & \beta(\lambda)  & 2\alpha(\lambda) & \cdots         & \vdots \\
 \vdots            & \vdots          & \ddots           & \ddots         & \beta(\lambda) \\
0                 & 0               & \cdots           & \beta(\lambda) & 2 \alpha(\lambda) 
\end{bmatrix}\,.
\]
whose spectrum is given by
\[
\Bigl\{\alpha(\lambda)+\beta(\lambda)\cos\frac{k\pi}{N}\Bigr\}_{k=1}^{N-1}.
\]
Hence $ \mathcal M_\lambda^0 $ has the same eigenvalues as $\mathcal M_\lambda$ except $\alpha -\beta$. All its eigenvalues are positive. Again we can verify for the energy  $(N/2)^2$ in this case  that \eqref{eq:BeCoMa} holds with $\Omega= S^1 \setminus \{0\}$ and the same partition as in  the previous case.

\begin{remark} 
A more general situation in one dimension, corresponding to an interval, is 
analyzed in~\cite{BBCCM}.
\end{remark}

\end{document}